\documentclass{amsart}

\usepackage{amsmath,latexsym,amssymb,amsthm,graphicx}
\usepackage{mathtools}
\usepackage{xfrac}
\usepackage{fix-cm}
\usepackage{hyperref,doi}
\usepackage{stmaryrd}
\usepackage{multirow}
\usepackage[all]{xy}
\usepackage{enumitem}
\setenumerate[1]{label=(\thesection.\arabic*)}

\numberwithin{equation}{section}

\newtheorem{theorem}{Theorem}[section]
\newtheorem{lemma}[theorem]{Lemma}
\newtheorem{claim}[theorem]{Claim}
\newtheorem{corollary}[theorem]{Corollary}
\newtheorem{proposition}[theorem]{Proposition}

\theoremstyle{definition}
\newtheorem{remark}[theorem]{Remark}
\newtheorem{remarks}[theorem]{Remarks}
\newtheorem{example}[theorem]{Example}
\newtheorem{definition}[theorem]{Definition}

\def\Z{\ensuremath{\mathbb{Z}}}

\def\R{\ensuremath{\mathbb{R}}}
\def\C{\ensuremath{\mathbb{C}}}
\def\H{\ensuremath{\mathbb{H}}}
\def\T{\ensuremath{\mathbb{T}}}
\def\Ab{\overline{A}}
\def\As{A^{\ast}}
\def\Abs{\Ab^{\ast}}
\def\m1{\multicolumn{1}}
\def\tr{\rule[-1mm]{0mm}{4.5mm}}
\def\A3{\tn{Alt}(3)}
\def\tA3{\tau \A3}
\def\B{\mathcal{B}}

\newcommand{\pa}[1]{\left(#1\right)}
\newcommand{\cpa}[1]{\left\{#1\right\}}

\newcommand{\tn}[1]{\textnormal{#1}}
\newcommand{\br}[1]{\left[#1\right]}
\newcommand{\fg}[1]{\left\langle #1\right\rangle}

\newcommand{\Int}[1]{\tn{Int} \, #1}

\newcommand{\p}[1]{\pi_1 (#1)}
\newcommand{\card}[1]{\left| #1 \right|}
\newcommand{\norm}[1]{\left\| #1 \right\|}

\font\cuf=cmtt8
\newcommand{\curl}[1]{{\cuf #1}}

\begin{document}
\title{Borromean rays and hyperplanes}

\author[J.~Calcut]{Jack S. Calcut}
\address{Department of Mathematics\\
         Oberlin College\\
         Oberlin, OH 44074}
\email{jcalcut@oberlin.edu}
\urladdr{\href{http://www.oberlin.edu/faculty/jcalcut/}{\curl{http://www.oberlin.edu/faculty/jcalcut/}}}

\author[J.~Metcalf-Burton]{Jules R. Metcalf-Burton}
\address{Department of Mathematics\\
         Oberlin College\\
         Oberlin, OH 44074}
\email{jmetcalf@oberlin.edu}

\author[T.~Richard]{Taylor J. Richard}
\address{Department of Mathematics\\
         Oberlin College\\
         Oberlin, OH 44074}
\email{trichard@oberlin.edu}

\author[L.~Solus]{Liam T. Solus}
\address{Department of Mathematics\\
         University of Kentucky,\hfill\break
         \indent Patterson Office Tower Room 722\\
         Lexington, KY 40504}
\email{liam.solus@uky.edu}
\urladdr{\href{http://www.ms.uky.edu/~solusl/}{\curl{http://www.ms.uky.edu/\textasciitilde solusl/}}}

\date{November 27, 2012}

\begin{abstract}
Three disjoint rays in $\R^3$ form \emph{Borromean rays} provided their union is knotted, but the union of any two components is unknotted.
We construct infinitely many Borromean rays, uncountably many of which are pairwise inequivalent.
We obtain uncountably many Borromean hyperplanes.
\end{abstract}

\maketitle

\section{Introduction}\label{introduction}

For proper, locally flat embeddings in $\R^n$, it is well-known that:
\begin{enumerate}\setcounter{enumi}{\value{equation}}
\item\label{rlt} A ray (= copy of $[0,\infty)$) knots if and only if $n=3$.
\item\label{hlt} A hyperplane (= copy of $\R^{n-1}$) knots if and only if $n=3$.
\setcounter{equation}{\value{enumi}}
\end{enumerate}
Both facts hold in the smooth, piecewise linear, and topological categories~\cite{cks}.
Fox and Artin discovered the first knotted ray~\cite{foxartin}.
The boundary of a closed regular neighborhood of any knotted ray is a knotted hyperplane.
For $n> 3$, fact \ref{hlt} is the Cantrell-Stallings hyperplane unknotting theorem, an enhancement of the famous Schoenflies theorem of
Mazur and Brown \cite{cks}, \hbox{\cite[p.~98]{rushing}}.
Embeddings in $\R^n$, $n\neq3$, of at most countably many rays or hyperplanes were recently classified by King, Siebenmann, and the first author~\cite{cks}.
In $\R^3$, no classification is known or even conjectured.\\

A ray or multiray $r\subset\R^3$ is \emph{unknotted} if and only if an automorphism of $\R^3$ carries $r$ to a union of radial rays.
Unknotted multirays with the same number of components are ambient isotopic~\cite[Lemma~4.1]{cks}.\\

A rich collection of knotted rays may be obtained from wild arcs.
Let $a\subset S^3$ be an arc with one wild (= non-locally flat) point $p$.
Consider $r=a-\cpa{p}$ in $\R^3=S^3-\cpa{p}$.
If $p$ is an endpoint of $a$, then $r$ is a knotted ray.
If $p$ is an interior point of $a$, then $r$ is a knotted, two component multiray.
Hence, in $\R^3$:
\begin{enumerate}\setcounter{enumi}{\value{equation}}
\item There exist infinitely many knot types of a ray~\cite{alfordball}.
\item There exist uncountably many knot types of a ray~\cite{mcpherson}.
\item There exist uncountably many knot types of two component multirays with unknotted components~\cite{foxharrold}.
\setcounter{equation}{\value{enumi}}
\end{enumerate}

A three component multiray $r\subset\R^3$ will be called \emph{Borromean rays} provided $r$ is knotted, but any two components of $r$ form an unknotted multiray.
Debrunner and Fox constructed an example equivalent to Borromean rays~\cite{debrunnerfox}.
Earlier, Doyle attempted a construction~\cite{doyle}, but his argument contained a gap~\cite{debrunnerfox}.
We prove that there exist uncountably many knot types of Borromean rays.
The following is an overview.\\

Consider the four blocks in Figure~\ref{four_blocks}.
\begin{figure}[h!]
    \centerline{\includegraphics[scale=1.0]{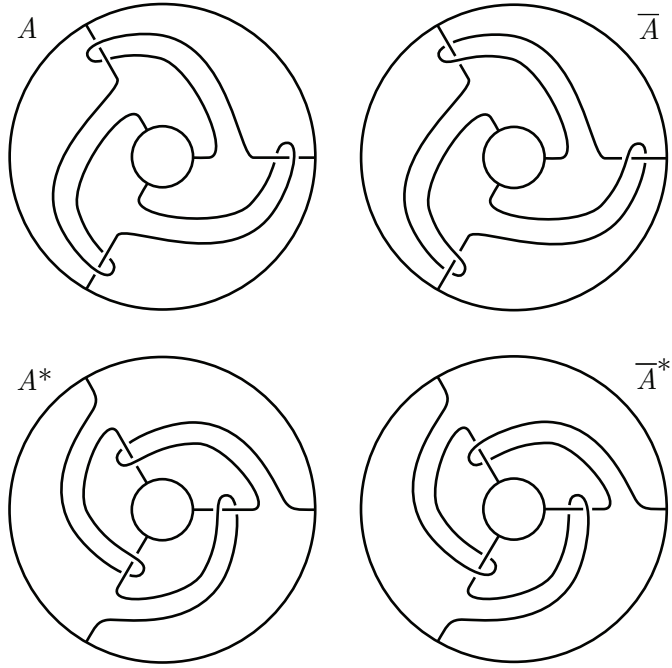}}
    \caption{Four blocks $A$, $\Ab$, $\As$, and $\Abs$. Each block is a three component tangle in a thickened $2$-sphere. The set of these four blocks is denoted by $\B$.}
\label{four_blocks}
\end{figure}
The block $A$ consists of a three component tangle $T$ in a thickened $2$-sphere $S^2\br{1,2}$.
Any two components of $T$ can be straightened by an ambient isotopy of $S^2\br{1,2}$ relative to boundary.
However, no diffeomorphism of $S^2\br{1,2}$ sends $T$ to a radial tangle (Corollary~\ref{Anottrivial}).
The blocks $\Ab$, $\As$, and $\Abs$ are reflections of $A$.
Let $\B=\cpa{A,\Ab,\As,\Abs}$ be the set of these four blocks.
Let $B_i$, $i\in\Z^{+}$, be a sequence of blocks in $\B$.
The \emph{infinite concatenation} $D^3 B_1 B_2 B_3 \cdots$ is obtained by gluing the inner boundary sphere of $B_1$ to the boundary of a $3$-disk,
and gluing the inner boundary sphere of $B_{i+1}$ to the outer boundary sphere of $B_i$ for each $i\in\Z^{+}$.
This yields the pair:
\begin{equation}\label{borr_tau}
	\pa{\R^3,\tau}=D^3 B_1 B_2 B_3 \cdots
\end{equation}
where $\tau\subset\R^3$ is a three component multiray.
Each such $\tau$ forms Borromean rays (Corollary~\ref{borr_blocks_yield_borr_rays}).
Let $\sigma$ be the Borromean rays determined by another such sequence $C_i$, $i\in\Z^{+}$.
We prove that if $f:\pa{\R^3,\tau}\to\pa{\R^3,\sigma}$ is a diffeomorphism of pairs,
then there is an isotopy of $f$ to a diffeomorphism $g:\pa{\R^3,\tau}\to\pa{\R^3,\sigma}$ and an integer $n$ such that:
\begin{equation}\label{block_respect}
	g\pa{B_i}=C_{i+n} \quad \tn{for all sufficiently large $i\in\Z^{+}$.}
\end{equation}
Hence, the existence of $f$ boils down to: (i) the tails of the sequences $B_i$ and $C_i$, and (ii) possible diffeomorphisms between individual blocks in $\B$.
The latter are studied in Section~\ref{s:diffeo_blocks}.
Our main result, Theorem~\ref{borr_rays_thm}, gives necessary and sufficient conditions for two such sequences to yield equivalent Borromean rays.
Care is taken to account for orientation.
As an application, we give necessary and sufficient conditions for our Borromean rays to be achiral (Corollary~\ref{chiral_cor}).
While most turn out to be chiral, we give a countably infinite family of pairwise inequivalent, achiral Borromean rays.\\

The notion of an \emph{irreducible block} plays a central role.
A block $B$ is \emph{irreducible} provided: if $B$ is diffeomorphic to a concatenation $B_1 B_2$, then $B_1$ or $B_2$ is diffeomorphic to a \emph{trivial block} (= block with a radial tangle).
Trivial blocks are irreducible (Proposition~\ref{sigma_en}).
We use this fact to prove that each $\tau$ in~\eqref{borr_tau} forms Borromean rays.
The block $A$ is also irreducible (Theorem~\ref{A_irred}), although the proof is more technical.
Thus, blocks in $\B$ are irreducible.
This fact is used to improve diffeomorphisms as in~\eqref{block_respect}.\\

We are unaware of a general method for detecting irreducibility.
For instance, let $B_1$ and $B_2$ be blocks containing $n$ component tangles $\tau_1$ and $\tau_2$ respectively.
Let $\tau$ be the tangle in the concatenation $B_1 B_2$.
Let $G_1$, $G_2$, and $G$ be the fundamental groups of $B_1 -\tau_1$, $B_2 -\tau_2$, and $B_1 B_2-\tau$ respectively.
Let $\Sigma$ be the $2$-sphere where $B_1$ and $B_2$ meet in $B_1 B_2$.
Then, $\Sigma'=\Sigma-\tau$ is an $n$-punctured sphere and $\pi_1\pa{\Sigma'}=F_{n-1}$ is free of rank $n-1$.
Using Dehn's Lemma and the Loop Theorem~\cite[p.~101]{rolfsen}, one may show that the inclusions $\Sigma'\hookrightarrow B_i-\tau_i$ induce injective homomorphisms on fundamental groups.
By van Kampen's theorem, $G=G_1 \ast_{F_{n-1}} G_2$ is the free product of $G_1$ and $G_2$ amalgamated over $F_{n-1}$ (see~\cite[\S4.2]{mks}).
By Grushko's theorem~\cite{stallings}, the rank of the free product $G_1 \ast G_2$ equals $rank \, G_1 + rank \, G_2$.
Thus, one might hope that $rank \, G \geq rank \, G_1 + rank \, G_2 - rank \, F_{n-1}$.
However, no such relation holds in general for free products with amalgamation~\cite[\S4]{weidmann2002}.
Still, rank behaves better when the amalgamating subgroup is malnormal in each factor~\cite{karrass_solitar}, \cite{weidmann2001}.
For knot groups, malnormality of the peripheral subgroups was studied recently by Weidmann~\cite{weidmann1998} and de la Harpe and Weber~\cite{dlharpe_weber}.
It is unclear to us whether $\pi_1\pa{\Sigma'}$ is malnormal in $G_1$ for an arbitrary block $B_1$.
It would be interesting to find block invariants sensitive to irreducibility.\\

We discovered the block $A$ as follows.
Consider a three component multiray $\tau\subset\R^3$ with the property:
\begin{enumerate}[label=(\dag)]
\item\label{two_comps_std} Any two components of $\tau$ form an unknotted multiray.
\end{enumerate}
Let $\tau_i$, $i\in\cpa{1,2,3}$, denote the components of $\tau$.
Property~\ref{two_comps_std} implies that for each pair $\tau_i$ and $\tau_j$ of components of $\tau$
there is a \emph{strip} $S_{i,j}\subset\R^3$ (= properly embedded copy of $[0,1]\times[0,\infty)$) whose \emph{stringers}
(= $\cpa{0}\times[0,\infty)$ and $\cpa{1}\times[0,\infty)$) equal $\tau_i$ and $\tau_j$.
The interior of $S_{i,j}$ probably intersects the third component of $\tau$ (if not, then $\tau$ is unknotted).
Using a small regular neighborhood of $\tau$, one may twist these strips about their stringers so that they patch together to form a general position immersion $f:S^1\times[0,\infty)\looparrowright \R^3$.
Let $p_i$, $i\in\cpa{1,2,3}$, be equally spaced points in $S^1$.
The immersion $f$ sends the radial ray $\cpa{p_i}\times[0,\infty)$ to $\tau_i$ for each $i\in\cpa{1,2,3}$ and is an embedding on each of the three closed sectors between such radial rays.
In an attempt to unknot $\tau$, one may try to eliminate multiple points of $f$.
Certain types of multiple points can be eliminated.
However, difficulties arise from essential circles of double points.
Figure~\ref{immersion} displays the relevant, compact part of a simple configuration where two essential circles of double points are identified under $f$.
\begin{figure}[h!]
    \centerline{\includegraphics[scale=1.0]{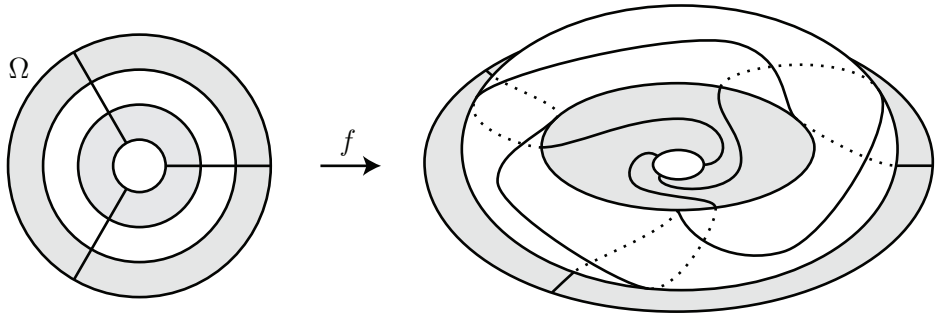}}
    \caption{Immersion $f$ from the compact, $2$-dimensional annulus $\Omega$ into the thickened sphere $S^2\br{1,2}$. The image of $f$ is a torus sitting atop an annulus and containing the tangle $T$.}
\label{immersion}
\end{figure}
Here, the domain of $f$ is the compact, $2$-dimensional annulus $\Omega$ containing three radial arcs.
The image of these three radial arcs is the tangle $T$ in the block $A$.
Having found $A$, a fundamental group calculation shows that $A$ is not a trivial block (see Section~\ref{equiv_rel_blocks}).
Our proof that $A$ is irreducible (see Section~\ref{AB_irred}) makes essential use of the immersion $f$ in Figure~\ref{immersion}.\\

After we discovered the block $A$, we found Debrunner and Fox's \emph{mildly wild $3$-frame} \cite{debrunnerfox}, \cite[pp.~95--97]{rushing}
and Doyle's attempted example~\cite{doyle}.
We pause to make some observations on these two examples.
\begin{enumerate}[label=(\arabic*),leftmargin=*]\setcounter{enumi}{0}
\item Debrunner and Fox's mildly wild $3$-frame $D\subset\R^3$ is compact, periodic, and contains one wild point $o$. In our notation, their building block is the concatenation $\Ab \ \Abs$.
Put $D$ in $S^3$. We define the \emph{Debrunner-Fox Borromean rays} to be $\delta=D-\cpa{o}$ in $\R^3=S^3-\cpa{o}$. In our notation:
\[
	\pa{\R^3,\delta}=D^3 A \As A \As A \As \cdots
\]
By Corollary~\ref{chiral_cor} below, $\delta$ is achiral.
\item Debrunner and Fox's proof that $D$ is wild hinges on showing a certain group is not finitely generated.
Their approach yields a mildly wild $n$-frame for each $n\geq 3$.
On the other hand, it is not clear how one can use it to distinguish between two wild $3$-frames.
In Section~\ref{irred_blocks}, we use irreducibility of trivial blocks to prove our multirays (including $\delta$) are knotted.
Then, we use irreducibility of blocks in $\B$ to distinguish between multirays.
\item Bing showed that Doyle's $3$-frame is standard~\cite{debrunnerfox}, though Bing's argument is not indicated.
Lemma~\ref{multiray_straightening} below is useful for recognizing unknotted multirays and applies to Doyle's $3$-frame.
\end{enumerate}

This paper is organized as follows.
Section~\ref{definitions} presents conventions and notation, introduces blocks (including several examples), and proves some basic properties concerning blocks.
Section~\ref{s:diffeo_blocks} studies diffeomorphisms between individual blocks in $\B$.
Section~\ref{irred_blocks} introduces irreducible blocks, proves trivial blocks are irreducible, deduces some corollaries, and constructs infinitely many irreducible blocks containing two component tangles.
Section~\ref{ball_arc_pairs} identifies some unknotted ball-arc pairs in blocks. 
Section~\ref{AB_irred} proves that blocks in $\B$ are irreducible.
Section~\ref{improve_spheres} simplifies certain spheres in concatenations of Borromean blocks and deduces two useful corollaries.
Section~\ref{brbh} classifies Borromean rays arising from sequences of blocks in $\B$ and then uses regular neighborhoods of multirays to obtain results on knotted multiple hyperplane embeddings.
In particular, we prove that there exist uncountably many pairwise inequivalent so-called \emph{Borromean hyperplanes} in $\R^3$.

\section{Building Blocks}\label{definitions}

We work in the smooth (= $C^{\infty}$) category.
Throughout, $\approx$ denotes diffeomorphism of manifolds or manifold pairs.
A map is \textbf{proper} provided the inverse image of each compact set is compact.
All isotopies will be smooth and proper.
A submanifold $X\subset M$ is \textbf{neat} provided $\partial X= X \cap \partial M$ and this intersection is transverse~\cite[pp.~30--31]{hirsch}, \cite[pp.~27, 31, 62]{kosinski}.\\

A \textbf{ray} is a proper embedding of $[0,\infty)$.
A \textbf{multiray} is a proper embedding of $Z\times [0,\infty)$ where $Z$ is a finite or countably infinite discrete space.
Indeed, each embedded submanifold of $\R^n$ contains at most countably many components since $\R^n$ is a separable metric space.
A ray in $\R^n$ is \textbf{radial} provided either it is straight and emanates from the origin, or it is contained in such a ray.
In particular, a radial ray can meet the origin only at its endpoint.
A collection of intervals embedded in $\R^n$ is \textbf{radial} provided each component lies in a radial ray.\\

The standard euclidean norm on $\R^n$ is $\left\|x\right\|:=\pa{\sum x_i^2}^{1/2}$.
On $\R^3$, the euclidean norm function will be denoted:
\begin{equation}\begin{split}\label{euclidean_norm}
\xymatrix@R=0pt{
	\mathbb{R}^3	\ar[r]^-{\eta}		&		\mathbb{R}\\
	x				\ar@{|-{>}}[r]	&		\left\|x\right\|}
\end{split}
\end{equation}
All lengths come from the standard euclidean metric.
The unit $n$-disk $D^n$ consists of all points $p\in\R^n$ such that $\left\|p\right\|\leq 1$.
The unit $(n-1)$-sphere is $S^{n-1}=\partial D^n$.
The sphere of radius $t>0$ in $\R^n$ about $0$ is denoted $S^{n-1}\br{t}$ and is called a \textbf{level sphere}.
In particular, $S^2=S^2\br{1}$.
Let $S^{n-1} \br{t_1,t_2}$, where $0<t_1<t_2$, denote the thickened sphere of points $p\in\R^n$ such that $t_1\leq\left\|p\right\|\leq t_2$.
In particular, $\partial S^{n-1} \br{t_1,t_2}$ equals the disjoint union of the spheres $S^{n-1}\br{t_1}$ and $S^{n-1}\br{t_2}$.
Let $S^{n-1}[t,\infty)$ denote the half-infinite annulus of points $p\in\R^n$ such that $\left\|p\right\|\geq t$.

\subsection{Blocks}\label{blocks}

Let $S^2 \br{t_1,t_2}$, where $0<t_1<t_2$, be a thickened sphere in $\R^3$.
A \textbf{tangle} $\tau$ is an embedding of the disjoint union of $n\geq1$ copies of $[0,1]$ as a neat submanifold of $S^2 \br{t_1,t_2}$.
If $\tau_k$ is a component of $\tau$, then the initial point $p_k$ of $\tau_k$ must lie in $S^2\br{t_1}$ and the terminal point must equal $(t_2/t_1)p_k\in S^2\br{t_2}$.
So, $\tau_k$ stretches between the two boundary $2$-spheres of $S^2\br{t_1,t_2}$, and its initial and terminal points lie on a radial ray.\\
 
A \textbf{block} is a pair $(S^2\br{t_1,t_2},\tau)$ where $\tau$ is a tangle.
Each block is oriented: $S^2\br{t_1,t_2}$ inherits its orientation from the standard one on $\R^3$,
and each component of $\tau$ is oriented to point out from the inner boundary $2$-sphere.
A \textbf{diffeomorphism} of blocks is any diffeomorphism of the corresponding pairs of spaces, not necessarily orientation or boundary preserving in any sense.\\

A \textbf{trivial block} is any pair $\varepsilon_n$ consisting of an $n$ component, radial tangle in a thickened sphere (see Figure~\ref{trivial_block}).\\

\begin{figure}[h!]
    \centerline{\includegraphics[scale=1.0]{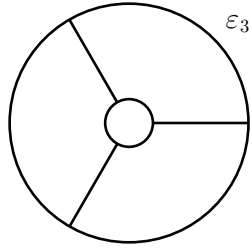}}
    \caption{Trivial block $\varepsilon_3$ containing a three component, radial tangle.}
\label{trivial_block}
\end{figure}

By our convention, in every displayed block, the positive $x$-axis points horizontally to the right, the positive $y$-axis points vertically up,
and the positive $z$-axis points out of the page towards the reader.\\

Given any block $B=(S^2\br{t_1,t_2},\tau)$, define two blocks:
\begin{enumerate}[itemsep=3pt]\setcounter{enumi}{\value{equation}}
\item $\overline{B}:=	(S^2\br{t_1,t_2},\overline{\tau})$ is the reflection of $B$ across the $xy$-plane.
\item $B^{\ast}:=	(S^2\br{t_1,t_2},\tau^{\ast})$ is the inversion of $B$ across $S^2\br{(t_1+t_2)/2}$.
\setcounter{equation}{\value{enumi}}
\end{enumerate}
If $B=(S^2\br{1,2},\tau)$, then inversion is $p\mapsto (3-\left\|p\right\|)\frac{p}{\left\|p\right\|}$.
Components of $B^{\ast}$ are still oriented out from the inner boundary $2$-sphere.
Evidently, the bar and star operations commute and are involutions:
\[
	\overline{B}^{\ast}=\overline{B^{\ast}}, \quad \overline{\overline{B}}=B, \quad \tn{and} \quad \pa{B^{\ast}}^{\ast}=B. 
\]

Figure~\ref{four_blocks} above introduced four blocks important for our purposes.
Let:
\[
	\B=\cpa{A,\Ab,\As,\Abs}
\]
be the set of these four blocks.
Note that $\B$ is closed under the bar and star operations.
By construction, blocks in $\B$ are pairwise diffeomorphic.
They are pairwise distinct, though, up to finer equivalence relations, as explained below.

\subsection{Equivalence Relations on Blocks}\label{equiv_rel_blocks}

The coarsest equivalence relation on blocks we consider is that of \emph{diffeomorphism} (defined above).
Finer diffeomorphism relations, involving orientation and/or boundary preservation, arise in Section~\ref{s:diffeo_blocks}.
On blocks with the same underlying thickened spheres, the finest relation we consider (short of equality) is that of \textbf{ambient isotopy relative to boundary},
meaning ambient isotopy of tangles fixing both boundary $2$-spheres pointwise at all times.

\begin{example}[\textbf{One Component Tangles}]\label{one_comp}
Any block $B=\pa{S^2\br{1,2},\tau}$, where $\tau$ has one component, is ambient isotopic relative to boundary to $\varepsilon_1$.
This fact is well-known (e.g., it's an exercise in Rolfsen \cite[p.~257]{rolfsen}).
We are not aware of a published proof, so we sketch one.
All isotopies are ambient and relative to boundary.
By a preliminary isotopy (left to the reader), we assume $B$ appears as in Figure~\ref{knot_block_e1}
where $K$ is a crossing diagram in general position.
\begin{figure}[h!]
    \centerline{\includegraphics[scale=1.0]{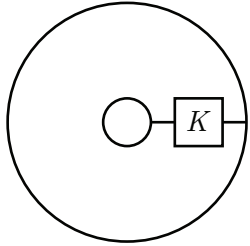}}
    \caption{Block $B$ with a one component tangle $\tau$.}
\label{knot_block_e1}
\end{figure}
It suffices to prove that crossings of $K$ may be switched by isotopy,
since then we may arrange that $K$ has monotonic $z$-coordinate and the result follows.
So, consider a crossing $C$ of $K$ with over arc $\alpha$ and under arc $\beta$.
Push $\alpha$ and $\beta$ sufficiently close together in the $z$-direction.
Let $p$ and $q$ be the midpoints of $\alpha$ and $\beta$ respectively, where $p$ lies directly above $q$.
Let $[0,t]$ denote the subarc of $\tau$ where $0\in S^2$ and $t\in\Int\tau$.
Let $J=[0,p]$ if $q\notin [0,p]$, and let $J=[0,q]$ if $q\in [0,p]$.
In other words, $J$ is the unique subarc of $\tau$ originating on $S^2$ and terminating at the first point, $p$ or $q$, encountered by $J$. 
Assume $J=[0,p]$ (otherwise, flip the picture over).
Isotop $q$ close to $S^2$ by following just underneath $J$ and stretching $\beta$.
Then, loop $\beta$ under $S^2$ and isotop $q$ back (again using $J$ as a guide) to lie above $p$.
The crossing $C$ has been switched, completing the proof.
\end{example}

\begin{remark}\label{knots_tied}
Example~\ref{one_comp} has the following possibly surprising corollary, which appears to be due to Wilder~\cite[p.~987]{foxartin}.
If the ray $r\subset\R^3$ is obtained by tying successive knots in a radial ray (see Figure~\ref{knots_ray}),
\begin{figure}[h!]
    \centerline{\includegraphics[scale=1.0]{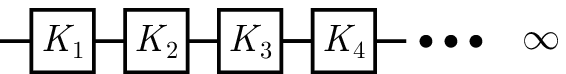}}
    \caption{Ray $r\subset\R^3$ obtained by tying successive knots in a radial ray.}
\label{knots_ray}
\end{figure}
then $r$ is ambient isotopic to a radial ray.
Proof: Let $K_i\subset S^2\br{i,i+1}$ for each $i\in\Z^+$.
Simultaneously apply the straightening process
from Example~\ref{one_comp} to each $S^2\br{i,i+1}$. $\square$
\end{remark}

\begin{example}[\textbf{Two Component Tangles}]\label{knot_blocks}
Let $k\subset S^3$ be a knot.
Let $D \subset S^3$ be a $3$-disk such that: (i) $a:=k\cap D$ is a neatly embedded arc in $D$, and (ii) $(D,a)$ is an unknotted ball-arc pair.
Let $D'\subset \Int{\varepsilon_2}$ be a small, round $3$-disk meeting one tangle component in an arc.
The \textbf{knot block} $B(k)$ is obtained from $\varepsilon_2$ by replacing $D'$ with $S^3-\Int{D}$ as in Figure~\ref{knot_block}.
\begin{figure}[h!]
    \centerline{\includegraphics[scale=1.0]{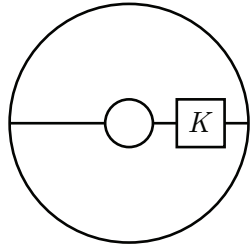}}
    \caption{Knot block $B\pa{k}$ where $K$ is a diagram yielding $k$.}
\label{knot_block}
\end{figure}
In general, $B(k)$ is well-defined up to diffeomorphism.
If $k$ itself is oriented, then one could define $B(k)$ more carefully.
Let $\tau$ be the tangle in a knot block $B(k)$.
Evidently, deleting the boundary from $B(k)-\tau$ yields $S^3 - k$.
In particular, $B(k)\approx B(k')$ implies $S^3-k \approx S^3-k'$.
So, knots with nonisomorphic groups (e.g., torus knots~\cite[p.~47]{bz}) yield nondiffeomorphic knot blocks.
Finally, let $B=\pa{S^2\br{1,2},\tau}$ be a block where $\tau$ has two components, $\tau_1$ and $\tau_2$.
Then, $B$ is ambient isotopic relative to boundary to some knot block.
To see this, straighten $\tau_2$ using the process in Example~\ref{one_comp}.
Then push $\tau_1$ away from $\tau_2$ by integrating a suitable vector field tangent to level spheres.
\end{example}

\begin{example}[\textbf{Dirac's Block}]\label{block_D}
Consider the block $D$ in Figure~\ref{one_twist}.
\begin{figure}[h!]
    \centerline{\includegraphics[scale=1.0]{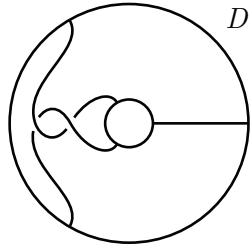}}
    \caption{Dirac's block $D$.}
\label{one_twist}
\end{figure}
The thickened sphere underlying $D$ is $S^2\br{1,2}$, and
$D$ is obtained from $\varepsilon_3$ by fixing the inner boundary $2$-sphere pointwise
and rigidly rotating the outer boundary sphere one revolution about the $x$-axis.
In particular, $\varepsilon_3 \approx D$, in fact by an orientation preserving diffeomorphism that is pointwise the identity on both boundary $2$-spheres.
On the other hand, $\varepsilon_3$ and $D$ are not ambient isotopic relative to boundary, as proved by Newman~\cite{newman} and Fadell~\cite{fadell} (see also~\cite{fadell_vanbuskirk} and~\cite[\S11.1--11.2]{murasugi_kurpita}).
If $D$ had been obtained from $\varepsilon_3$ by two complete twists, rather than just one,
then $D$ would have been ambient isotopic relative to boundary to $\varepsilon_3$ by Dirac's belt trick.
\end{example}

A diffeomorphism between thickened spheres is: (i) \textbf{radial} provided it sends radial arcs to radial arcs, and (ii)
\textbf{level} provided it sends level $2$-spheres to level $2$-spheres.
The next lemma says that there is essentially just one trivial block $\varepsilon_n$ for each $n\in\Z^+$.

\begin{lemma}[\textbf{Trivial Block Uniqueness}]\label{en_unique}
Let $\varepsilon_n=(S^2\br{t_1,t_2},r)$ and $\varepsilon'_n=(S^2\br{t'_1,t'_2},r')$ be trivial blocks.
Then, there is a radial, level, orientation preserving diffeomorphism $h:\varepsilon_n \to \varepsilon'_n$ sending $S^2\br{t_i}$ to $S^2\br{t'_i}$ for $i=1,2$.
If $\br{t_1,t_2}=\br{t'_1,t'_2}$, then there is an ambient isotopy $H_t$, $0\leq t\leq 1$, of $S^2\br{t_1,t_2}$ such that:
\begin{enumerate}\setcounter{enumi}{\value{equation}}
\item\label{eH0id} $H_0=\tn{Id}$ and $H_1=h$.
\item\label{H_tradial} $H_t$ is a radial, level diffeomorphism for all $0\leq t\leq 1$.
\setcounter{equation}{\value{enumi}}
\end{enumerate}
\end{lemma}

\begin{proof}
Let $\psi:\br{t_1,t_2}\to\br{t'_1,t'_2}$ be the unique affine, orientation preserving diffeomorphism.
Then, $x\mapsto \psi(\left\|x\right\|)\cdot x/\left\|x\right\|$ is a radial, level diffeomorphism $S^2\br{t_1,t_2}\to S^2\br{t'_1,t'_2}$.
So, it suffices to consider the case $\br{t_1,t_2}=\br{t'_1,t'_2}=[1,2]$.
Let $p_i\in S^2$, $1\leq i \leq n$, denote the initial points of the components $r_i$ of $r$.
Define $p'_i$ similarly for $r'$.
Let $\alpha_1$ be a smooth, simple path from $p_1$ to $p'_1$ in $S^2$.
Let $\nu\alpha_1$ be a smooth regular neighborhood of $\alpha_1$ in $S^2$.
There is an ambient isotopy of $S^2$, with support in $\nu\alpha_1$, carrying $p_1$ to $p'_1$.
For instance, begin with a suitable nonzero tangent vector field to $\alpha_1$, extend to a vector field $v$ on $S^2$ that vanishes outside of $\nu\alpha_1$, 
and then integrate $v$ (cf.~\cite[pp.~22--24]{milnor97}).
Extend this isotopy radially to get an ambient isotopy of $S^2\br{1,2}$ carrying $r_1$ to $r'_1$.
Any component $r_i$, $i\geq2$, that moved during this isotopy is still radial and is still denoted $r_i$.
Repeat this procedure, while choosing $\nu\alpha_i$ disjoint from $p'_1,\ldots,p'_{i-1}$. 
\end{proof}

To distinguish $A$ from $\varepsilon_3$ up to diffeomorphism, it suffices to distinguish the fundamental groups of their tangle complements up to isomorphism.
Presentations of such groups are obtained using Wirtinger's algorithm.
Consider the diagram of $A$ in Figure~\ref{wirt_A}.
\begin{figure}[h!]
    \centerline{\includegraphics[scale=1.0]{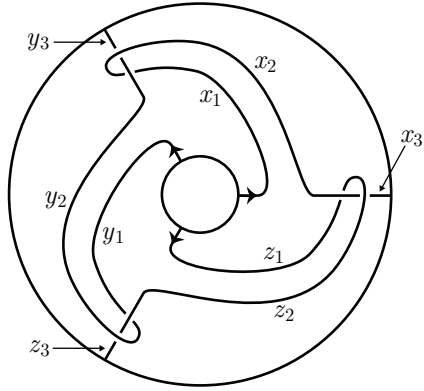}}
    \caption{Block $A$ with oriented and labeled arcs.}
\label{wirt_A}
\end{figure}
As usual, labels of arcs correspond to generators of $\pi_1(A-T)$.
The basepoint is above the page.
The based loop representing a generator $g$ first penetrates the plane of the page at a point just to the right of the oriented arc labeled $g$, and has linking number $+1$ with this oriented arc.
A presentation of $\p{A-T}$ is:
\[
	\fg{{\begin{array}{c} x_1,x_2,x_3,\\ y_1,y_2,y_3,\\ z_1,z_2,z_3 \end{array} } \Bigg|
	    {\begin{array}{c} y_2 x_2=x_1 y_2,\, x_2 y_3=y_2 x_2,\, x_2 z_2=z_1 x_2,\\
										    z_2 x_3=x_2 z_2,\, z_2 y_2=y_1 z_2,\, y_2 z_3=z_2 y_2,\\
                        x_1 y_1 z_1=1 \end{array} }}
\]
Each crossing of $T$ contributes a relation.
The last relation is evident topologically.
It is a \emph{vertex relation} for the fundamental group of the complement of the graph obtained by crushing the inner boundary $2$-sphere to a point \hbox{\cite[\S4]{whitehead}}, \hbox{\cite[p.~981]{foxartin}}, \hbox{\cite[p.~148]{stillwell}}.
The outer vertex relation, $x_3 y_3 z_3 =1$, is redundant.\\

For the trivial tangle, $\p{\varepsilon_3 -r}\cong F_2$ is free of rank $2$.
To distinguish $\pi_1(A-T)$ from $F_2$, we count their \emph{classes} of homomorphisms into small symmetric groups $\tn{Sym}(n)$ using the computer algebra system MAGMA (a finite problem).
Two homomorphisms $h_1,h_2:G\to \tn{Sym}(n)$ are considered \emph{equivalent} provided there exists $\pi\in\tn{Sym}(n)$ such that $h_2(g)=\pi^{-1} h_1(g)\pi$ for all $g\in G$.
Table~\ref{magma_data_1} collects this data.
\begin{table}[h!]\renewcommand{\arraystretch}{1.2}
\begin{center}
\begin{tabular}{c|c|c|}
\cline{2-3}
& $\p{\varepsilon_3 -r}$ & $\p{A-T}$ \\ \hline
\m1{|c|}{\tn{Sym}(1)} & $1$ & $1$  \\ \hline
\m1{|c|}{\tn{Sym}(2)} & $4$ & $4$ \\ \hline
\m1{|c|}{\tn{Sym}(3)} & $11$ & $11$ \\ \hline
\m1{|c|}{\tn{Sym}(4)} & $43$ & $47$  \\ \hline
\m1{|c|}{\tn{Sym}(5)} & $161$ & $193$  \\ \hline
\m1{|c|}{\tn{Sym}(6)} & $901$ & $1317$  \\ \hline
\end{tabular}
\end{center}
\vspace{3mm}
\caption{Numbers of classes of homomorphisms into $\tn{Sym}(n)$.}
\label{magma_data_1}
\end{table}
The two abelian symmetric groups are included for completeness, and the last two rows are included for comparision with Table~\ref{magma_data_2} ahead.
Any of the last three rows in Table~\ref{magma_data_1} implies the following.

\begin{corollary}\label{Anottrivial}
$A\not\approx \varepsilon_3$.
As blocks in $\B$ are pairwise diffeomorphic, no block in $\B$ is diffeomorphic to $\varepsilon_3$.
\end{corollary}

\subsection{Borromean Blocks}\label{borr_blocks}

A \textbf{Borromean block} is a block $B=(S^2\br{t_1,t_2},\tau)$ such that: (i) $B\not\approx\varepsilon_3$,
and (ii) each block obtained from $B$ by forgetting one component of $\tau$ is ambient isotopic relative to boundary to $\varepsilon_2$.
By Figure~\ref{four_blocks}, each block in $\B$ clearly satisfies condition (ii).
Corollary~\ref{Anottrivial} implies the following.

\begin{corollary}\label{four_blocks_borromean}
The blocks $A$, $\Ab$, $\As$, and $\Abs$ are Borromean.
\end{corollary}

It is useful to observe that any block whose tangle components always ``head out'' is diffeomorphically trivial.
Let $M$ be a smooth closed manifold.
A \textbf{tangle} $\tau$ in $M\times [0,1]$ is a neat embedding of a disjoint union of copies of $[0,1]$.
In particular, $\partial\tau\subset M\times\cpa{0,1}$ and $\tau$ has finitely many components.
In analogy with our definition of \emph{tangle in a thickened sphere}, we might also require each component of $\tau$
to have one boundary point in $M\times\cpa{0}$ and one in $M\times\cpa{1}$.
However, this property will automatically be satisfied in the following, our only use of tangles in general $M\times[0,1]$.

\begin{lemma}[\textbf{Tangle Straightening}]\label{tangle_straightening}
Let $M$ be a smooth closed manifold. Let:
\[
	p:M\times [0,1] \to [0,1]
\]
be projection.
Let $\tau\subset M\times[0,1]$ be a tangle such that $\left.p\right|\tau$ has no critical points.
Then, there is an ambient isotopy $H_t$, $0\leq t\leq 1$, of $M\times[0,1]$ such that:
\begin{enumerate}\setcounter{enumi}{\value{equation}}
\item\label{h0id} $H_0=\tn{Id}$.
\item\label{hrelinnerbound} $\left. H_t\right| M\times\cpa{0}=\tn{Id}$ for all $0\leq t\leq 1$.
\item\label{phtp} $p H_t = p$ for all $0\leq t \leq 1$.
\item\label{h1taustraight} $H_1(\tau)$ is a finite disjoint union of straight arcs $\cpa{x_i}\times[0,1]$, $x_i\in M$.
\setcounter{equation}{\value{enumi}}
\end{enumerate}
\end{lemma}

\begin{proof}
Let $u$ be a nonzero tangent vector field on $\tau$.
Since $\left.p\right|\tau$ has no critical points, we may assume $u(p)>0$ on $\tau$.
Extend $u$ to a small tubular neighborhood $U$ of $\tau$ so that $u(p)>0$ on $U$.
Then, $w:=u/u(p)$ is tangent to $\tau$, and $w(p)=1$ on $U$.
Using $U$, we get a smooth function $g:M\times[0,1]\to [0,1]$ with support in $U$ and equal to $1$ on $\tau$.
Consider the vector field $v:=gw+(1-g)(0,1)$ on $M\times[0,1]$, where $(0,1)$ is the obvious constant vector field on $M\times[0,1]$.
Let $\phi((x,s),t)$ be the maximal flow generated by $v$ and let $\mathcal{D}$ be the domain of $\phi$.
Then, $p\phi((x,s),t)=s+t$ on $\mathcal{D}$ since $v(p)=1$ on $M\times[0,1]$.
As $M$ has no boundary and $M\times[0,1]$ is compact, patching together local flows yields:
\[
\mathcal{D}=M\times\cpa{(s,t) \ | \ 0\leq s\leq 1, \ -s\leq t\leq 1-s}
\]
Let $q:M\times[0,1]\to M$ be projection.
Define the ambient isotopy $H_t$, $0\leq t\leq 1$, of $M\times[0,1]$ by:
\[
	H_t(x,s):=( q \phi((x,s),-st),s)
\]
That is, at time $t$, $(x,s)$ flows via $\phi$ from $M\times\cpa{s}$ back into $M\times\cpa{s(1-t)}$ and then is translated into $M\times\cpa{s}$.
Properties~\ref{h0id}--\ref{phtp} evidently hold, and tangency of $v$ to $\tau$ guarantees~\ref{h1taustraight}.
\end{proof}

\begin{remark}
Our proof of Lemma~\ref{tangle_straightening} is a modification of the proof of the isotopy extension theorem (cf.~\cite[Lemma~9.15]{cks}).
\end{remark}

\begin{corollary}[\textbf{Tangle Radialization}]\label{trivial_block_cor}
Let $B=(S^2\br{t_1,t_2},\tau)$ be a block where $\tau$ has $n\geq1$ components.
Assume $\left.\eta\right|\tau$ has no critical points.
Then, there is an ambient isotopy $G_t$, $0\leq t \leq 1$, of $S^2\br{t_1,t_2}$ such that:
\begin{enumerate}\setcounter{enumi}{\value{equation}}
\item $G_0=\tn{Id}$.
\item $\left. G_t\right| S^2\br{t_1}=\tn{Id}$ for all $0\leq t\leq 1$.
\item $\eta G_t = \eta$ on $S^2\br{t_1,t_2}$ for all $0\leq t\leq 1$. In particular, each $G_t$ is a level diffeomorphism.
\item $G_1(\tau)$ is radial. In particular, $G_1(B)$ equals a trivial block $\varepsilon_n$.
\setcounter{equation}{\value{enumi}}
\end{enumerate}
\end{corollary}

\begin{proof}
Let $l:\br{t_1,t_2}\to\br{0,1}$ be the unique affine, orientation preserving diffeomorphism.
Let $h:S^2\br{t_1,t_2}\to S^2\times[0,1]$ be the diffeomorphism given by $h(x):=\pa{x/\left\|x\right\|,l\pa{\left\|x\right\|}}$.
Apply Lemma~\ref{tangle_straightening} to the tangle $h(\tau)$ in $S^2\times[0,1]$ and let $H$ be the resulting isotopy of $S^2\times[0,1]$.
The desired isotopy is $G:=h^{-1}\circ H \circ (h\times\tn{Id})$.
\end{proof}

\begin{remark}
Consider a tangle $\tau$ in a Borromean block. Critical points of $\left.\eta\right|\tau$ come in pairs since, by definition, tangle components 
originate on the inner boundary sphere and terminate on the outer boundary sphere.
So, Corollary~\ref{trivial_block_cor} implies $\left.\eta\right|\tau$ has at least two critical points.
This minimum number is achievable.
A simple ambient isotopy relative to boundary of $A$ yields the (necessarily Borromean) block $A'$ in Figure~\ref{A_2_cp}.
\begin{figure}[h!]
    \centerline{\includegraphics[scale=1.0]{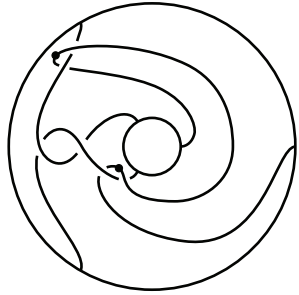}}
    \caption{Borromean block $A'=(S^2\br{1,2},T')$ ambient isotopic relative to boundary to $A$. The restriction $\left.\eta\right|T'$ has just the two indicated critical points.}
\label{A_2_cp}
\end{figure}
\end{remark}

The next corollary will be used when concatenating blocks.

\begin{corollary}[\textbf{Tangle Radialization Near Boundary}]\label{trnb}
Consider a block $B=(S^2\br{t_1,t_2},\tau)$ where $\tau$ has $n\geq1$ components.
There exists $\epsilon_0 >0$ such that, for each $0<\epsilon\leq\epsilon_0$, there is an ambient isotopy $K_t$, $0\leq t \leq 1$, of $S^2\br{t_1,t_2}$ such that:
\begin{enumerate}\setcounter{enumi}{\value{equation}}
\item $K_0=\tn{Id}$.
\item $K_t =\tn{Id}$ on $\partial S^2\br{t_1,t_2}$ and on $S^2\br{t_1+\epsilon,t_2-\epsilon}$ for all $0\leq t\leq 1$.
\item $\eta K_t=\eta$ on $S^2\br{t_1,t_2}$ for all $0\leq t \leq 1$. In particular, each $K_t$ is a level diffeomorphism.
\item $K_1(\tau)$ is radial in $S^2\br{t_1,t_1+\epsilon/2}$ and in $S^2\br{t_2-\epsilon/2,t_2}$.
\setcounter{equation}{\value{enumi}}
\end{enumerate}
\end{corollary}

\begin{proof}
By definition of a block, $\tau$ is neatly embedded.
So, there is $\epsilon_0>0$ such that $\epsilon_0<(t_1+t_2)/2-t_1$ and $\left.\eta\right|\tau$ has no critical points in
$S^2\br{t_1,t_1+\epsilon_0}\cup S^2\br{t_2-\epsilon_0,t_2}$.
Let $0<\epsilon\leq\epsilon_0$.
Recall the diffeomorphism $h:S^2\br{t_1,t_2}\to S^2\times[0,1]$ from the proof of Corollary~\ref{trivial_block_cor}.
Consider the tangle $h(\tau)\subset S^2\times[0,1]$.
Define $\epsilon':=\epsilon/(t_2-t_1)$ and note that $0<\epsilon'<1/2$.
Construct a smooth function $b:[0,1]\to[-\epsilon',\epsilon']$ such that:
\[
	b(s) = 
	\begin{cases}
	s			&	\tn{if } 0\leq s \leq \epsilon'/2\\
	0			&	\tn{if } \epsilon'\leq s \leq 1-\epsilon'\\
	s-1		&	\tn{if } 1-\epsilon'/2\leq s \leq 1
	\end{cases}
\]
Also, to ensure well-definition of the isotopy to come, arrange that:
\[
	b(s)\in
	\begin{cases}
		[0,s]		&	\tn{if } \epsilon'/2\leq s \leq \epsilon' \\
		[s-1,0]	&	\tn{if } 1-\epsilon'\leq s \leq 1-\epsilon'/2
	\end{cases}
\]
Return now to the proof of Lemma~\ref{tangle_straightening} and redefine:
\[
	H_t(x,s):=( q \phi((x,s),-b(s)t),s)
\]
This $H$ is an ambient isotopy of $S^2\times[0,1]$.
Then, $K:=h^{-1}\circ H \circ (h\times\tn{Id})$ is the desired isotopy.
\end{proof}

The following lemma can be useful for recognizing unknotted multirays.

\begin{lemma}[\textbf{Multiray Straightening}]\label{multiray_straightening}
Let $r$ be a multiray in $\R^3$ with at most countably many components.
By a small perturbation of $r$, assume $\left.\eta\right|r$ is a Morse function.
Suppose $\left.\eta\right|r$ has no critical points (i.e., the rays always ``head out'').
Then, $r$ is ambient isotopic to a radial multiray.
\end{lemma}

\begin{proof}
The proof of Lemma~\ref{tangle_straightening} adapts readily (cf.~\cite[Lemma~9.15]{cks}).
\end{proof}

\begin{remark}
By Remarks~\ref{concat_remarks} item~\ref{tail_det_type} below,
Lemma~\ref{multiray_straightening} holds provided $\left.\eta\right|r_i$ has only finitely many critical points for each component $r_i$ of $r$.
\end{remark}

\subsection{Concatenations of Blocks}

We define finite and countably infinite \textbf{concatenations} of blocks.
In any concatenation of blocks, two natural properties are always assumed:
\begin{enumerate}\setcounter{enumi}{\value{equation}}
\item Block summands contain tangles with the same number, $n\geq1$, of components.
\item Initial points of tangles in different summands differ by a radial dilation.
\setcounter{equation}{\value{enumi}}
\end{enumerate}

First, consider a finite concatenation $B_1 B_2 \cdots B_k$ of blocks $B_i$.
By definition, this means the block $B:=(S^2\br{1,k+1},\tau)$ where $B_i$ has been stretched radially to coincide with $S^2\br{i,i+1}$.
The resulting tangle $\tau$ in $B$ is the union of the $k$ stretched tangles.
Evidently, this definition is deficient in that $\tau$ may have corners where blocks meet.
Our remedy is to refine the definition: each block $B_i$ is first adjusted, by an ambient isotopy
relative to boundary provided by Corollary~\ref{trnb}, so that its tangle is radial near the boundary spheres of $B_i$.
The concatenation $B_1 B_2 \cdots B_k$ is now a block, well-defined up to (obvious) ambient isotopy relative to boundary.\\

The proof of the next lemma is straightforward.

\begin{lemma}\label{e_nidentity}
If $B$ is a block, then $\varepsilon_n B \approx B \approx B\varepsilon_n$.
\end{lemma}

Each infinite concatenation $B_1 B_2 B_3\cdots$ is defined analogously.
The result is $(S^2 [1,\infty),\tau)$ where $\tau$ is now a neatly embedded $n$ component multiray.
Further, $D^3 B_1 B_2 B_3\cdots$ denotes the pair $(\R^3,\tau)$ where $\tau$ is the obvious $n$ component multiray in $\R^3$.
A \textbf{diffeomorphism} between infinite concatenations is a diffeomorphism between the corresponding pairs of spaces.
An \textbf{ambient isotopy} between infinite concatenations is an ambient isotopy of the underlying total space carrying one multiray to the other.
An ambient isotopy of $D^3 B_1 B_2 B_3\cdots$ is thus an ambient isotopy of $\R^3$, and it need not fix $D^3$.\\

\begin{remarks}\label{concat_remarks}
\noindent\begin{enumerate}[label=(\arabic*),leftmargin=*]\setcounter{enumi}{0}
\item\label{tail_det_type} Both the diffeomorphism type and the ambient isotopy type of any infinite concatenation of blocks $D^3 B_1 B_2 B_3 \cdots$ are determined by any tail
$B_m B_{m+1} B_{m+2} \cdots$ of $B_1 B_2 B_3 \cdots$.
This is a consequence of the fact that the multiray $\tau$ determined by $D^3 B_1 B_2 B_3 \cdots$ may be
shrunk outwards.
Using a regular neighborhood $\nu \tau$ of $\tau$ in $\R^3$, one may construct an ambient isotopy of $\R^3$,
with support in $ \Int \nu \tau$, that sends each component $\tau_k$ of $\tau$ into itself at all times
and slides the initial point $\partial \tau_k$ up along $\tau_k$ any given finite amount.
\item Let $f_i : B_i\to C_i$ be diffeomorphisms of blocks where $1\leq i\leq k$ (a similar remark applies to the case $i\in\Z^+$).
If each $f_i$ is pointwise the identity on boundary $2$-spheres, then evidently we get a diffeomorphism of the concatenations 
$B_1 B_2 \cdots B_k \to C_1 C_2 \cdots C_k$ by pasting together diffeomorphisms (and smoothing the resulting homeomorphism near boundary $2$-spheres~\cite[p.~182]{hirsch}).
\item If $B_1\approx C_1$ and $B_2 \approx C_2$, then, in general, one may not conclude that $B_1 B_2$ and $C_1 C_2$ are diffeomorphic.
For instance, let $f_1=\tn{Id}:A\to A$, and let $f_2:A\to\Abs$ be the diffeomorphism from Table~\ref{obvious_homeos} below;
both diffeomorphisms preserve orientation, $f_1$ is pointwise the identity everywhere, and $f_2$ preserves boundary spheres setwise.
However, by Table~\ref{magma_data_2} below, $AA\not\approx A\Abs$.
\end{enumerate}
\end{remarks}

We close this section with a special instance where one may conclude that $B_1 B_2\approx C_1 C_2$.

\begin{lemma}\label{pw_id_b}
Let $f:B\to\varepsilon_n$ be a diffeomorphism.
Then, there is a trivial block $\varepsilon'_n$ and diffeomorphisms $g,k:B\to\varepsilon'_n$ such that $g$ is pointwise the identity on outer boundary $2$-spheres
and $k$ is pointwise the identity on inner boundary $2$-spheres.
\end{lemma}

\begin{proof}
Without loss of generality, $B$ and $\varepsilon_n$ have underlying thickened sphere $S^2\br{1,2}$.
Composing $f$ with inversion across $S^2\br{\sfrac{3}{2}}$ if necessary, we assume $f$ preserves boundary $2$-spheres setwise.
Define $h:S^2\br{1,2}\to S^2\br{1,2}$ by:
\[
h(p)=\frac{\norm{p}}{2}f^{-1}\pa{2\frac{p}{\norm{p}}}
\]
That is, $h=f^{-1}$ on $S^2\br{2}$, and $h$ is extended radially to $S^2\br{1,2}$.
So, $\varepsilon'_n :=h(\varepsilon_n)$ is a trivial block and $g:=hf$ is the desired diffeomorphism.
The other case is similar.
\end{proof}

\begin{lemma}\label{diff_impl_diff}
Let $B_1\approx C_1$ and $B_2 \approx C_2$. Suppose one of the blocks $B_1$, $B_2$, $C_1$, or $C_2$ is diffeomorphic to $\varepsilon_n$.
Then, $B_1 B_2\approx C_1 C_2$.
\end{lemma}

\begin{proof}
Without loss of generality, the four given blocks have underlying thickened spheres $S^2\br{1,2}$.
Assume there is a diffeomorphism $f:B_1 \to \varepsilon_n$ (the other cases are similar).
By Lemma~\ref{pw_id_b}, there is a diffeomorphism $g: B_1 \to \varepsilon'_n$ that is pointwise the identity on $S^2\br{2}$.
We have diffeomorphisms:
\[
 B_1 B_2 \to \varepsilon'_n B_2 \to B_2
\]
where the first is given by pasting together $g$ and $\tn{Id}$, and the second is given by Lemma~\ref{e_nidentity}.
As $C_1\approx B_1\approx \varepsilon_n$, we similarly obtain diffeomorphisms:
\[
 C_1 C_2 \to \varepsilon''_n C_2 \to C_2
\]
As $B_2\approx C_2$, the result follows.
\end{proof}

\section{Diffeomorphisms Between Blocks in \texorpdfstring{$\B$}{B}}\label{s:diffeo_blocks}

Let $\tn{Hom}(\B)$ denote the set of all diffeomorphisms $h:B_1\to B_2$ such that $B_1$ and $B_2$ lie in $\B$.
The category $\mathcal{C}$ with objects $\B$ and with morphisms $\tn{Hom}(\B)$ is a \emph{groupoid} (i.e., every morphism is an isomorphism).
The composition $h_2\circ h_1$ of elements $h_1,h_2\in\tn{Hom}(\B)$ exists if and only if the codomain block of $h_1$ equals the domain block of $h_2$.
Given a block in $\B$, we number its tangle components by $1$, $2$, $3$, beginning with the component having boundary points on the $x$-axis and proceeding counterclockwise (CCW).\\

To each $h:B_1\to B_2$ in $\tn{Hom}(\B)$, we associate:
\begin{enumerate}[label=(\arabic*)]\setcounter{enumi}{0}
\item The domain block $B_1$.
\item The codomain block $B_2$.
\item The \emph{orientation character}: $+1$ if $h$ preserves orientation of $S^2\br{1,2}$, and $-1$ otherwise.
\item The \emph{boundary character}: $+1$ if $h$ preserves the boundary $2$-spheres componentwise, and $-1$ if $h$ swaps them.
\item The \emph{tangle permutation}: $\sigma=\sigma(h)\in\tn{Sym}(3)$, defined using the numbering of tangle components.
\end{enumerate}
The \textbf{type} of $h$ is the $5$-tuple of these parameters.
The set of possible types is:
\[
	U:=\B\times\B\times \cpa{\pm1} \times \cpa{\pm1} \times \tn{Sym}(3)
\]
The set $U$  has cardinality $384$ and has natural operations:
\begin{align*}
 (B_1,B_2,\varepsilon,\delta,\sigma)^{-1}:= &(B_2,B_1,\varepsilon,\delta,\sigma^{-1})\\
 (B_2,B_3,\varepsilon_2,\delta_2,\sigma_2)\circ (B_1,B_2,\varepsilon_1,\delta_1,\sigma_1):= &(B_1,B_3,\varepsilon_1 \varepsilon_2,\delta_1 \delta_2,\sigma_2 \sigma_1)
\end{align*}
The former is a unary operation defined on all types, and the latter is a partial composition.
We compose permutations right to left, as with functions; MAGMA does the reverse.
The category $\mathcal{D}$ with objects $\B$ and morphisms $U$ is also a groupoid; if $B\in\B$, then the \textbf{identity morphism} is $1_B:=(B,B,+1,+1,\tn{Id})$.
We have a functor of groupoids $F:\mathcal{C} \to\mathcal{D}$.
It is the identity on objects and is defined on morphisms by:
\[
\xymatrix@R=0pt{
	\tn{Hom}(\B)	\ar[r]^-{F}			&	U	\\
	h   \ar@{|-{>}}[r]  & \tn{type}(h)}
\]

We wish to determine the types realized by diffeomorphisms, i.e., the image of $F$ on morphisms.
Table~\ref{hom_block} presents exactly these types, where $\tn{A}_3 := \fg{(1,2,3)}$ is the alternating group of degree $3$
and $C:=(1,2)\tn{A}_3$ is the coset of transpositions.
The remainder of this section is devoted to proving Table~\ref{hom_block}.\\

\begin{table}[h!]
\begin{center}
\begin{tabular}{cc|cc|cc|cc|cc|}
\cline{3-10}
& & \multicolumn{2}{|c|}{\rule[0mm]{0mm}{4.5mm}$A$} &  \multicolumn{2}{|c|}{$\Ab$} & \multicolumn{2}{|c|}{$\As$} & \multicolumn{2}{|c|}{$\Abs$}\\
\cline{3-10}
& & \m1 {|c|}{$+1$} & \m1{|c|}{\tr$-1$} & \m1{|c|}{$+1$} & \m1{|c|}{$-1$} & \m1{|c|}{$+1$} & \m1{|c|}{$-1$} & \m1{|c|}{$+1$} & \m1{|c|}{$-1$} \\ \hline
\m1{|c|}{\multirow{2}{*}{$A$}} & \m1{|c|}{\tr$+1$} & $\tn{A}_3$ & & & $\tn{A}_3$ & & $C$ & $C$ & \\
\cline{2-2}
\m1{|c|}{} & \m1{|c|}{\tr$-1$} & $C$ & & & $C$ & & $\tn{A}_3$ & $\tn{A}_3$ & \\ \hline
\m1{|c|}{\multirow{2}{*}{$\Ab$}} & \m1{|c|}{\tr$+1$} & & $\tn{A}_3$ & $\tn{A}_3$ & & $C$ & & & $C$ \\
\cline{2-2}
\m1{|c|}{} & \m1{|c|}{\tr$-1$} & & $C$ & $C$ & & $\tn{A}_3$ & & & $\tn{A}_3$ \\ \hline
\m1{|c|}{\multirow{2}{*}{$\As$}} & \m1{|c|}{\tr$+1$} & & $C$ & $C$ & & $\tn{A}_3$ & & & $\tn{A}_3$ \\
\cline{2-2}
\m1{|c|}{} & \m1{|c|}{\tr$-1$} & & $\tn{A}_3$ & $\tn{A}_3$ & & $C$ & & & $C$\\ \hline
\m1{|c|}{\multirow{2}{*}{$\Abs$}} & \m1{|c|}{\tr$+1$} & $C$ & & & $C$ & & $\tn{A}_3$ & $\tn{A}_3$ & \\
\cline{2-2}
\m1{|c|}{} & \m1{|c|}{\tr$-1$} & $\tn{A}_3$ & & & $\tn{A}_3$ & & $C$ & $C$ & \\ \hline
\end{tabular}
\end{center}
\vspace{3mm}
\caption{Types of diffeomorphisms of blocks in $\B$.
Columns indicate domain and orientation character.
Rows indicate codomain and boundary character.
A given subset of $\tn{Sym}(3)$ specifies tangle permutations realized by such a diffeomorphism.
An empty entry means no such diffeomorphism exists.}
\label{hom_block}
\end{table}

Table~\ref{obvious_homeos} lists some obvious diffeomorphisms of blocks and their types.
To see the last two diffeomorphisms, perform the rotation on the given domain block, then perform a simple ambient isotopy relative to boundary.\\

\begin{table}[h!]
\begin{center}
\begin{tabular}{|l|c|c|}
\hline
\tr Diffeomorphism	& Domain/Codomain	&	type \\ \hline
\tr Identity	&	$B\to B$, $B\in\B$ & $(B,B,+1,+1,\tn{Id})$ \\ \hline
\tr CCW Rotation by	& \multirow{2}{*}{$B\to B$, $B\in\B$} & \multirow{2}{*}{$(B,B,+1,+1,(1,2,3))$} \\
$2\pi/3$ about $z$-axis	&	&\\ \hline
\rule[0mm]{0mm}{4.5mm}Reflection across & $A\to \Ab$     & $(A,\Ab,-1,+1,\tn{Id})$ \\
$xy$-plane            & $\As\to \Abs$  & $(\As,\Abs,-1,+1,\tn{Id})$ \\ \hline
\tr Inversion across & $A\to \As$     & $(A,\As,-1,-1,\tn{Id})$ \\
intermediate $S^2$    & $\Ab \to \Abs$ & $(\Ab,\Abs,-1,-1,\tn{Id})$ \\ \hline
\rule[0mm]{0mm}{4.5mm}Rotation by $\pi$ & $A\to \Abs$    & $(A,\Abs,+1,+1,(2,3))$ \\
about $x$-axis        & $\Ab\to \As$   & $(\Ab,\As,+1,+1,(2,3))$ \\ \hline
\end{tabular}
\end{center}
\vspace{3mm}
\caption{Diffeomorphisms of blocks.}
\label{obvious_homeos}
\end{table}

The image of $F$ is a subgroupoid of $\mathcal{D}$, with objects $\B$ and morphisms $F(\tn{Hom}(\B))$.
In particular, $F(\tn{Hom}(\B))$ is closed under taking inverses and forming well-defined compositions.
So: (i) take the types in Table~\ref{obvious_homeos}, (ii) include their inverses, and (iii) further include all well-defined binary compositions.
This yields the $96$ realizable types in Table~\ref{hom_block}.
It remains to exclude the other types.\\

Exclusion of types will utilize fundamental groups of certain tangle complements.
If $B$ is a block, or a concatenation of blocks, and $\tau$ is the tangle therein, then $B-\tau$ denotes the complement of $\tau$ in the thickened sphere.
As explained in Section~\ref{equiv_rel_blocks}, we count classes of homomorphisms into small symmetric groups $\tn{Sym}(n)$ using MAGMA.
Table~\ref{magma_data_2} collects this data.
\begin{table}[h!]\renewcommand{\arraystretch}{1.2}
\begin{center}
\begin{tabular}{c|c|c|c|c|}
\cline{2-5}
&  $\rule[0mm]{0mm}{4.5mm}\p{AA-\tau}$ & $\p{A\Ab-\tau}$ & $\p{A\As-\tau}$ & $\p{A\Abs-\tau}$\\ \hline
\m1{|c|}{\tn{Sym}(1)} &  $1$ &  $1$ &  $1$ &  $1$ \\ \hline
\m1{|c|}{\tn{Sym}(2)} &  $4$ &  $4$ &  $4$ &  $4$ \\ \hline
\m1{|c|}{\tn{Sym}(3)} &  $11$ &  $11$ &  $11$ &  $11$ \\ \hline
\m1{|c|}{\tn{Sym}(4)} &  $63$ &  $63$ &  $63$ &  $63$ \\ \hline
\m1{|c|}{\tn{Sym}(5)} &  $342$ &  $342$ &  $354$ &  $330$ \\ \hline
\m1{|c|}{\tn{Sym}(6)} &  $3111$ &  $3255$ &  $3525$ &  $3105$ \\ \hline
\end{tabular}
\end{center}
\vspace{3mm}
\caption{Numbers of classes of homomorphisms into $\tn{Sym}(n)$.}
\label{magma_data_2}
\end{table}

\begin{lemma}\label{non_exist_types}
Let $B_1 \neq B_2$ be blocks in $\B$. Then, the type $(B_1,B_2,+1,+1,\tn{Id})$ is not in the image of $F$.
\end{lemma}

\begin{proof}
Suppose otherwise.
Then, there is a diffeomorphism $h:B_1\to B_2$ of type $(B_1,B_2,+1,+1,\tn{Id})$.\\

Let $S_{0,3}$ denote the $2$-sphere $S^2$ with three marked points $p_1,p_2,p_3$.
The mapping class group $\tn{Mod}(S_{0,3})$ is the group of orientation preserving diffeomorphisms of $S^2$ that send $\cpa{p_1,p_2,p_3}\to\cpa{p_1,p_2,p_3}$,
modulo isotopies of $S^2$ fixing $p_1$, $p_2$, and $p_3$ at all times~\cite[pp.~44--45]{farbmargalit}.
Recall that the natural map
\begin{equation}\label{3ps}
\xymatrix{
	\tn{Mod}(S_{0,3}) \ar[r]^{\cong} & \tn{Sym}(3)}
\end{equation}
which sends an element of $\tn{Mod}(S_{0,3})$ to its action on the three marked points, is an isomorphism~\cite[\S2.2.2]{farbmargalit}.\\

By hypothesis, $h$ preserves boundary $2$-spheres componentwise and preserves orientation of $S^2\br{1,2}$.
So, $\left.h\right| S^2$ is an orientation preserving diffeomorphism.
As the tangle permutation of $h$ is the identity,~\eqref{3ps} permits us to assume $h$ is the identity on inner boundary $2$-spheres.
Pasting together the diffeomorphisms $\tn{Id}:A\to A$ and $h:B_1\to B_2$ yields a diffeomorphism $AB_1\to AB_2$.
This contradicts the last row of Table~\ref{magma_data_2} since $B_1 \neq B_2$.
\end{proof}

Consider two types $\alpha,\beta\in U$ such that $\beta\circ\alpha$ is defined.
As the image of $F$ is a subgroupoid of $\mathcal{D}$,
if any two of $\alpha$, $\beta$, or $\beta\circ\alpha$ lie in the image of $F$, then the third does as well.
Therefore, if $\alpha$ is any type forbidden by Lemma~\ref{non_exist_types}, $\beta$ is any of the already realized $96$ types in the image of $F$, and $\beta\circ\alpha$ is defined in $\mathcal{D}$, then $\beta\circ\alpha$ is not in the image of $F$.
A tedious, but completely straightforward calculation (facilitated by MAGMA), shows that this yields $288$ types not in the image of $F$.
This completes our proof of Table~\ref{hom_block}, and yields the following.

\begin{corollary}\label{blocks_inequiv_isotopy}
The blocks $A$, $\Ab$, $\As$, and $\Abs$ are pairwise distinct up to ambient isotopy relative to boundary.
\end{corollary}

\section{Irreducible Blocks}\label{irred_blocks}

This section introduces the notion of an \emph{irreducible block}.
Such blocks play a central role in our construction of Borromean rays.
First, recall the following standard definitions and accompanying lemma.\\

A $2$-sphere $\Sigma$ in a thickened sphere is \textbf{essential} provided it does not bound a $3$-disk in the thickened sphere.
Otherwise, $\Sigma$ is \textbf{inessential}.
Let $a$ be an arc transverse to a $2$-manifold $\Sigma$ in a $3$-manifold.
Then, $\card{a\cap\Sigma}$ denotes the number of points in $a\cap\Sigma$ (ignoring any orientations).
The \textbf{mod $2$ intersection number} of $a$ and $\Sigma$, denoted $\#_2(a,\Sigma)$,
is $\card{a\cap\Sigma}\mod 2$.

\begin{lemma}\label{ess_sphere}
Let $a$ be a neatly embedded arc in $S^2\br{t_1,t_2}$.
Assume $a$ has one boundary point in $S^2\br{t_1}$ and the other in $S^2\br{t_2}$.
Let $\Sigma$ be a $2$-sphere embedded in the interior of $S^2\br{t_1,t_2}$ and transverse to $a$.
The following are equivalent: (i) $\Sigma$ is essential in $S^2\br{t_1,t_2}$,
(ii) $\#_2(a,\Sigma)\ne0$,
and (iii) there is a neighborhood $U$ of $\partial S^2\br{t_1,t_2}$ in $S^2\br{t_1,t_2}$ and an ambient isotopy $H_t$,
$0\leq t \leq 1$, of $S^2\br{t_1,t_2}$ such that:
\begin{enumerate}\setcounter{enumi}{\value{equation}}
\item $H_0=\tn{Id}$.
\item $\left.H_t\right|U=\tn{Id}$ for all $0\leq t\leq 1$.
\item $H_1(\Sigma)$ is a level $2$-sphere in $\Int{S^2\br{t_1,t_2}}$.
\setcounter{equation}{\value{enumi}}
\end{enumerate}
\end{lemma}

\begin{proof}
Assume, without loss of generality, that $\br{t_1,t_2}=[1,2]$.
Let $D$ denote \hbox{$D^3\cup S^2\br{1,2}$}, the $3$-disk of radius $2$.
By the $3$-dimensional Schoenflies theorem \cite[Ch.~III]{cerf}, \cite[\S17]{moise}, \cite[Thm.~1.1]{hatcher3d}, $\Sigma$ bounds a unique $3$-disk $\Delta\subset\tn{Int}D$.
Let $Y:=D-\tn{Int}\Delta$. So, $D=\Delta\cup Y$ and $\Delta\cap Y=\Sigma$.\newline
Case 1. $D^3\subset \tn{Int}Y$.
Then, $\Delta$ lies in the interior of $S^2[1,2]$.
So, $\Sigma$ is inessential and clearly $\#_2(a,\Sigma)=0$.\newline
Case 2. $D^3\subset \tn{Int}\Delta$.
Then, $\Sigma$ cannot bound a $3$-disk in $S^2[1,2]$ and, hence, is essential in $S^2[1,2]$.
The arc $a$ has one boundary point in $\Int{\Delta}$ and one outside $\Delta$, so $\#_2(a,\Sigma)=1$.
Let $\Sigma'$ be a level $2$-sphere between $\Sigma$ and $S^2\br{2}$.
Let $D'$ be the $3$-disk in $\Int{D}$ with boundary $\Sigma'$.
Let $R:=D'-\Int{\Delta}$ be the compact region in $\Int{D}$ with boundary $\Sigma\sqcup\Sigma'$.
By uniqueness of disk embeddings~\cite[p.~185]{hirsch}, there is an ambient isotopy of $D'$ carrying $\Delta$ to a round $3$-disk.
Hence, there is a diffeomorphism $g:R\to S^2\times[0,1]$ sending $\Sigma$ to $S^2\times\cpa{0}$.
Construct a vector field $v$ on $S^2[1,2]$ as follows.
On $R$, $v$ is the pushforward by $g^{-1}$ of the constant vector field $(0,1)$ on $S^2\times[0,1]$.
Extend $v$ to the rest of $S^2\br{1,2}$, making it $0$ outside a small neighborhood of $R$.
The isotopy generated by $v$ is the desired $H_t$.
\end{proof}

Next, we give two definitions of \textbf{irreducible block} and then we prove they are equivalent.
Let $B=(S^2\br{t_1,t_2},\tau)$ be a block where $\tau$ has $n\geq1$ components.

\begin{definition}[\textbf{Irreducible Block, First Definition}]\label{irred_def_1}
The block $B$ is \textbf{irreducible} provided: if $B$ is diffeomorphic to a concatenation of blocks $B_1 B_2$,
then $B_1$ or $B_2$ (or both) is diffeomorphic to a trivial block $\varepsilon_n$.
\end{definition}

\begin{definition}[\textbf{Irreducible Block, Second Definition}]\label{irred_def_2}
Let $\Sigma$ be a $2$-sphere embedded in the interior of $S^2\br{t_1,t_2}$ and transverse to $\tau$.
Assume $\Sigma$ meets each component of $\tau$ at exactly one point.
The block $B$ is \textbf{irreducible} provided:
there is a neighborhood $U$ of
$\partial S^2\br{t_1,t_2}$ in $S^2\br{t_1,t_2}$ and an ambient isotopy $H_t$, $0\leq t \leq 1$, of $S^2\br{t_1,t_2}$ such that:
\begin{enumerate}\setcounter{enumi}{\value{equation}}
\item\label{H0id} $H_0=\tn{Id}$.
\item\label{HrelU} $\left.H_t\right|U=\tn{Id}$ for all $0\leq t\leq 1$.
\item\label{r_setwise} $H_t$ fixes $\tau$ setwise for all $0\leq t \leq 1$.
\item $H_1(\Sigma)$ is a level $2$-sphere in $S^2\br{t_1,t_2}$.
\item $\left.\eta\right|\tau$ has no critical points between $H_1(\Sigma)$ and $S^2\br{t_1}$, or between $H_1(\Sigma)$ and $S^2\br{t_2}$.
\setcounter{equation}{\value{enumi}}
\end{enumerate}
\end{definition}

\begin{proposition}\label{defs_equiv}
The two definitions of irreducible block are equivalent.
\end{proposition}

\begin{proof}
Assume $B$ is irreducible according to the second definition.
Suppose $h:B\to B_1 B_2$ is a diffeomorphism.
Assume $h$ sends the inner boundary $2$-sphere to the inner boundary $2$-sphere, the other case being similar.
Let $\Sigma'$ be the level $2$-sphere in $B_1 B_2$ along which the concatenation takes place.
Let $\Sigma:=h^{-1}(\Sigma')$.
Then, $\Sigma$ satisfies the hypotheses of the second definition.
Let $H_t$ be the isotopy provided by the second definition.
So, $H_1(\Sigma)$ is a level $2$-sphere in $S^2\br{t_1,t_2}$ and, say,
$\left.\eta\right|\tau$ has no critical points in the compact region $R$ with boundary $H_1(\Sigma)\sqcup S^2\br{t_1}$.
By Corollary~\ref{trivial_block_cor}, the block $(R,\tau\cap R)$ is diffeomorphic to a trivial block $\varepsilon_n$.
Hence, $B_1\approx\varepsilon_n$, as desired.\\

Next, assume $B$ is irreducible according to the first definition.
Let $\Sigma$ satisfy the hypotheses of the second definition.
By Lemma~\ref{ess_sphere}, $\Sigma$ is essential in $S^2\br{t_1,t_2}$, and there is an isotopy of $S^2\br{t_1,t_2}$
(probably disturbing $\tau$) which carries $\Sigma$ to a level $2$-sphere, $\Sigma'$.
Let $\tau'$ be the image of $\tau$ under this isotopy.
This $\Sigma'$ divides $(S^2\br{t_1,t_2},\tau')$ into two obvious blocks $B_1$ and $B_2$.
Evidently, $B$ is diffeomorphic to the concatenation $B_1 B_2$.
The first definition of irreducible block implies that, say, $B_1$ is diffeomorphic to $\varepsilon_n$.
By Lemma~\ref{en_unique} (Trivial Block Uniqueness), we can and do assume $\varepsilon_n$ has underlying thickened sphere $S^2[1,2]$.
Let $R$ denote the compact region in $S^2\br{t_1,t_2}$ between $\Sigma$ and $S^2\br{t_1}$.
Hence, there is a diffeomorphism of pairs $g:(R,\tau\cap R)\to\varepsilon_n$, and $g$ sends $S^2\br{t_1}$ to $S^2$.
Let $\Sigma'':=S^2\br{t_1+\alpha}$ for some small $\alpha>0$ to be specified.
First, choose $\alpha$ small enough so $\Sigma''$ lies between $\Sigma$ and $S^2\br{t_1}$.
Next, reduce $\alpha$ if necessary so that:
\begin{enumerate}\setcounter{enumi}{\value{equation}}
\item\label{nocp} $\left.\eta\right|\tau$ has no critical points on or between $\Sigma''$ and $S^2\br{t_1}$.
\setcounter{equation}{\value{enumi}}
\end{enumerate}
This reduction is possible since $\tau$ is neatly embedded.
Note that futher reducing $\alpha$ maintains condition~\ref{nocp}.
Finally, reduce $\alpha$ if necessary so that:
\begin{enumerate}\setcounter{enumi}{\value{equation}}
\item\label{nottangent} $g(\Sigma'')$ is nowhere tangent to any radial arc in $S^2\br{1,2}$.
\setcounter{equation}{\value{enumi}}
\end{enumerate}
This last reduction is possible since: (i) $g$ is a diffeomorphism, (ii) $g(S^2\br{t_1})=S^2$, and (iii) $S^2\br{t_1}$ is compact.
As $g(\Sigma'')$ is essential in $S^2\br{1,2}$, condition~\ref{nottangent} implies that:
\begin{enumerate}\setcounter{enumi}{\value{equation}}
\item\label{onepointtransverse} Each radial arc in $S^2\br{1,2}$ of length $1$ intersects $g(\Sigma'')$ in exactly one point and transversely.
\setcounter{equation}{\value{enumi}}
\end{enumerate}
Condition~\ref{onepointtransverse} permits construction of an ambient isotopy of $S^2\br{1,2}$ that carries $g(\Sigma'')$ to $S^2\br{3/2}$ and merely slides points along radial arcs.
So, by an abuse of notation, we further assume the diffeomorphism of pairs $g$ itself sends $\Sigma''$ to $S^2\br{3/2}$.
Let $R'$ be the compact region in $S^2\br{t_1,t_2}$ between $\Sigma$ and $\Sigma''$.
Construct a vector field $v$ on $S^2\br{t_1,t_2}$ as follows.
On $R'$, $v$ is the pushforward by $g^{-1}$ of the vector field $-p/(2\left\|p\right\|)$ on $S^2\br{3/2,2}$.
Note that $v$ is tangent to $\tau$ on $R'$.
Extend $v$ to the rest of $S^2\br{t_1,t_2}$, making it $0$ outside a small neighborhood of $R'$ and ensuring tangency to $\tau$.
The isotopy generated by $v$ is the desired $H_t$.
\end{proof}

In general, it appears to be a difficult problem to decide whether a given block is irreducible.
We prove next that trivial blocks are irreducible and then observe some corollaries.

\begin{proposition}\label{sigma_en}
Each trivial block $\varepsilon_n=(S^2\br{t_1,t_2},\tau)$, $n\geq1$, is irreducible.
\end{proposition}

\begin{proof}
Let $\Pi$ denote the annulus where $S^2\br{1,2}$ meets the $xy$-plane.
Lemma~\ref{en_unique} reduces us to the case where $\varepsilon_n$ has underlying thickened sphere $S^2\br{1,2}$
and $\tau$ consists of $n$ equally spaced radial arcs in $\Pi$.
We prove $\varepsilon_n$ is irreducible according to the first definition.
By the first paragraph of the proof of Proposition~\ref{defs_equiv}, it suffices to consider a $2$-sphere, $\Sigma$,
embedded in the interior of $S^2\br{1,2}$, transverse to $\tau$, and intersecting each component $\tau_i$ of $\tau$
in one point $p_i$.
Let $X\subset S^2\br{1,2}$ be the compact set between $\Sigma$ and $S^2$.
It suffices to produce a diffeomorphism $g:\pa{X,X\cap\tau}\to \varepsilon_n$.\\

We will improve $\Sigma$ (and, hence, $X$) by ambient isotopies of $S^2\br{1,2}$.
Improved spaces will be denoted by their original names,
except $\Pi$ always denotes $S^2\br{1,2} \cap \tn{($xy$-plane)}$.
So, assume $\Sigma$ intersects $\Pi$ transversely.
Thus, $\Sigma\cap\Pi$ is a closed $1$-manifold and one component, $K$, of $\Sigma\cap\Pi$ must contain all of the points $p_i$.\\

If $\Sigma \cap \Pi\neq K$, then consider a component, $C$, of $\Sigma\cap \Pi$ that is innermost in its component of $\Sigma - K$.
Let $D_1$ be the $2$-disk in $\Sigma - K$ with boundary $C$.
Let $D_2$ be the $2$-disk in $\Pi$ with boundary $C$.
Then, $D_1\cup D_2$ is an embedded $2$-sphere in the interior of $S^2\br{1,2}$ disjoint from $\tau$.
By Lemma~\ref{ess_sphere}, $D_1\cup D_2$ is inessential in $S^2\br{1,2}$.
Let $D$ be the $3$-disk in $S^2\br{1,2}$ with boundary $D_1\cup D_2$.
This $D$ permits construction of an isotopy of $S^2\br{1,2}$, with support near $D$, that carries $D_1$ past $D_2$ to a parallel copy of $D_2$.
Thus, $C$ (at least) has been eliminated from $\Sigma \cap \Pi$.
Repeating this operation finitely many times, we get $\Sigma \cap \Pi = K$.\\

Now, we give a bootstrapping definition of the required diffeomorphism $g$.
First, $g$ sends $S^2\to S^2$ by the identity.
Second, $g$ sends $X\cap\tau_i$ to $\tau_i$ by an affine diffeomorphism for each $i$.
Third, $g$ sends $K$ to $S^2\br{2}\cap\Pi$.
Fourth, $g$ sends a smooth, regular neighborhood of $\pa{S^2\cap\Pi}\cup \pa{X\cap\tau}\cup K$ in $X\cap\Pi$ to
a smooth, regular neighborhood of $\pa{S^2\cap\Pi} \cup\tau \cup \pa{S^2\br{2}\cap\Pi}$ in $\Pi$.
This step may be accomplished, quite concretely, by judiciously choosing (closed) collars~\cite[\S4.6]{hirsch} and ambiently rounding corners.
Fifth, $g$ sends $X\cap\Pi$ to $\Pi$.
By the smooth $2$-dimensional Schoenflies theorem \cite[Remark~9.19]{cks}, this step evidently requires extension of $g$ over $n$ smooth $2$-disks $\Delta_i \subset X\cap\Pi$, $i=1,\ldots,n$.
Let $R_i\subset\Pi$ denote the $2$-disk with boundary $g\pa{\partial\Delta_i}$.
As every diffeomorphism of $S^1$ extends to one of $D^2$, each diffeomorphism $\left.g\right| \partial\Delta_i$ extends to a diffeomorphism $\psi_i:\Delta_i\to R_i$.
Extending $g$ over $\Delta_i$ by $\psi_i$ yields a well-defined homeomorphism (smooth except possibly at $\partial \Delta_i$).
By an isotopy of $\psi_i$, relative to $\partial \Delta_i$ and with support in a collar of $\partial \Delta_i$~\cite[p.~182]{hirsch},
this extension is a diffeomorphism.
Sixth, $g$ sends $\Sigma\to S^2\br{2}$.
Seventh, $g$ sends a smooth, regular neighborhood of $S^2\cup \pa{X\cap\Pi}\cup \Sigma$ in $X$ to
a smooth, regular neighborhood of $S^2 \cup\Pi \cup S^2\br{2}$ in $S^2\br{1,2}$.
This is done as in step four (product a nice corner rounding with $S^1$).
Finally, $g$ sends $X$ to $S^2\br{1,2}$.
By the smooth $3$-dimensional Schoenflies theorem, this step requires extension of $g$ over two smooth $3$-disks.
This is done as in step five, except using the fact that every diffeomorphism of $S^2$ extends to one of $D^3$ \cite{munkres60}, \cite{smale}, \cite[pp.~202--206]{thurston}.
By construction, the diffeomorphism $g:X\to S^2\br{1,2}$ sends $X\cap\tau$ to $\tau$.
\end{proof}

\begin{corollary}\label{cor1}
Let $B_i$, $1\leq i\leq k$, be any blocks.
If $B_1 B_2 \cdots B_k \approx \varepsilon_n$, then $B_i\approx\varepsilon_n$ for each $1\leq i\leq k$.
\end{corollary}

\begin{proof}
By induction, it suffices to consider the case $k=2$.
The case $k=2$ follows from the proof of Proposition~\ref{sigma_en}, since $X$ can be the compact region between $\Sigma$ and $S^2$, or between $\Sigma$ and $S^2\br{2}$.
Alternatively, it is instructive to see that the case $k=2$ follows from the statement of Proposition~\ref{sigma_en} as follows.
Proposition~\ref{sigma_en} implies that $B_1\approx\varepsilon_n$ or $B_2\approx\varepsilon_n$.
Assume $B_1\approx\varepsilon_n$ (the other case is similar).
We have diffeomorphisms:
\[
	\varepsilon_n \approx B_1 B_2 \approx \varepsilon_n B_2 \approx B_2
\]
where the first exists by hypothesis, the second follows from Lemma~\ref{diff_impl_diff} since $B_1\approx\varepsilon_n$ and $B_2\approx B_2$, and
the last is given by Lemma~\ref{e_nidentity}.
\end{proof}

\begin{corollary}\label{cor2}
Let $B_i$, $i\in\Z^+$, be any blocks.
If $B_1 B_2 B_3 \cdots \approx \varepsilon_n \varepsilon_n \varepsilon_n \cdots$, then $B_i\approx \varepsilon_n$ for each $i\geq 1$.
\end{corollary}

\begin{proof}
Let $f:B_1 B_2 B_3 \cdots \to \varepsilon_n \varepsilon_n \varepsilon_n \cdots$ be a diffeomorphism (of pairs).
Let $\tau$ and $r$ be the $n$-component multirays determined by these concatenations respectively.
Clearly, $r$ is radial.
Isotopies of $f$ will send $\tau$ to $r$ at all times.
Note that $f\pa{S^2}=S^2$ since $f$ must restrict to a diffeomorphism on the boundaries of the total spaces.
Let $\Sigma:=f\pa{S^2\br{2}}$.
By a radial isotopy, relative to $S^2$, we may assume $\Sigma\subset \Int{S^2\br{1,2}}$.
This $\Sigma$ satisfies the hypotheses in the second definition of irreducible block.
Propositions~\ref{defs_equiv} and~\ref{sigma_en} permit us to isotop $f$, relative to $S^2$, so that $\Sigma=S^2\br{2}$.
Having isotoped $f$ so that $f\pa{S^2\br{i}}=S^2\br{i}$ for $1\leq i\leq k$ and some $k\geq 2$,
the same argument permits us to further isotop $f$, relative to $S^2\br{1,k}$, so that $f\pa{S^2\br{k+1}}=S^2\br{k+1}$.
Evidently, the composition of all of these (infinitely many) isotopies is a well-defined, smooth, proper isotopy.
So, we can and do assume $f\pa{S^2\br{k}}=S^2\br{k}$ for all $k\in\Z^+$.
The result is now immediate.
\end{proof}

\begin{corollary}\label{cor3}
Let $B_i$, $i\in\Z^+$, be any blocks.
If $D^3 B_1 B_2 B_3 \cdots \approx D^3 \varepsilon_n \varepsilon_n \varepsilon_n \cdots$, then $B_i\approx \varepsilon_n$ for all sufficiently large $i\in\Z^{+}$.
\end{corollary}

\begin{proof}
Let $f:D^3 B_1 B_2 B_3 \cdots \to D^3 \varepsilon_n \varepsilon_n \varepsilon_n \cdots$ be a diffeomorphism (of pairs).
Let $\tau$ and $r$ be the $n$-component multirays determined by these concatenations respectively.
Isotopies of $f$ will send $\tau$ to $r$ at all times.
By compactness, there exists $k\geq 2$ such that $f\pa{S^2\br{k}}$ is disjoint from $D^3$.
As in the previous proof, we may isotop $f$, relative to $D^3$, so that $f\pa{S^2\br{k}}=S^2\br{k}$.
Restricting $f$ to $S^2[k,\infty)$ implies $B_k B_{k+1} B_{k+2} \cdots \approx \varepsilon_n \varepsilon_n \varepsilon_n \cdots$.
Now, apply the previous corollary.
\end{proof}

Borromean blocks were defined in Section~\ref{borr_blocks}.
A multiray $\tau\subset\R^3$ forms \textbf{Borromean rays} provided: (i) no diffeomorphism of $\R^3$ carries $\tau$ to a radial multiray,
and (ii) each multiray obtained from $\tau$ by forgetting one component is ambiently isotopic to a two component, radial multiray.

\begin{corollary}\label{borr_blocks_yield_borr_rays}
Let $B_i$, $i\in\Z^+$, be Borromean blocks and let:
\[
	\pa{\R^3,\tau}=D^3 B_1 B_2 B_3 \cdots
\]
Then, $\tau$ forms Borromean rays.
In particular, the conclusion holds if each $B_i \in \B$.
\end{corollary}

\begin{proof}[Proof of Corollary~\ref{borr_blocks_yield_borr_rays}]
By definition of Borromean block, $B_i\not\approx\varepsilon_3$.
Corollary~\ref{cor3} implies that $D^3 B_1 B_2 B_3 \cdots \not\approx D^3 \varepsilon_3 \varepsilon_3 \varepsilon_3\cdots$.
Hence, no diffeomorphism of $\R^3$ carries $\tau$ to a radial multiray.
Next, let $\sigma$ be obtained from $\tau$ by forgetting any one component.
Let $C_i$ be obtained from $B_i$ by forgetting the corresponding tangle component.
By definition of Borromean block, each $C_i$ is ambient isotopic (relative to boundary) to $\varepsilon_2$.
Performing these isotopies, for $i\in\Z^{+}$, simultaneously yields an ambient isotopy of $\R^3$ carrying $\sigma$ to a radial multiray.
Thus, $\tau$ forms Borromean rays.
Lastly, blocks in $\B$ are Borromean by Corollary~\ref{four_blocks_borromean}.
\end{proof}

\begin{remarks}
\noindent\begin{enumerate}[label=(\arabic*),leftmargin=*]\setcounter{enumi}{0}
\item Corollary~\ref{cor3} reduces the \emph{infinite} problem of constructing knotted multirays to the \emph{finite} problem of constructing nontrivial blocks. For instance, the infinitely generated group theory in~\cite{debrunnerfox} may be replaced by finitely generated group theory (as used in Section~\ref{equiv_rel_blocks} above).
\item The converse of Corollary~\ref{cor3} holds by Remarks~\ref{concat_remarks} item~\ref{tail_det_type}.
\end{enumerate}
\end{remarks}

We close this section by constructing infinitely many irreducible blocks containing two component tangles.
Recall the notion of a knot block from Example~\ref{knot_blocks}.

\begin{lemma}\label{pk_ib}
If $k\subset S^3$ is a prime knot, then $B(k)$ is an irreducible block.
\end{lemma}

\begin{proof}
Let $B(k)=\pa{S^2\br{1,2},\tau}$ where $\tau_2$ is radial and $\tau_1$ contains the diagram $K$ for $k$ as in Figure~\ref{knot_block}.
Let $\Sigma$ be a $2$-sphere embedded in the interior of $S^2\br{1,2}$, transverse to $\tau$, and intersecting each component $\tau_i$ of $\tau$
in one point $p_i$.
Perturb $\Sigma$ so it coincides with the level sphere through $p_2$ near $\tau_2$.
All isotopies will be ambient and relative to a neighborhood of both $\partial S^2[1,2]$ and $\tau_2$.
Subsets that move will be called by their original names.
As in the proof of Proposition~\ref{sigma_en}, we isotop $\Sigma$ to the level sphere containing $p_2$.
Push $\tau_1$ away from $\tau_2$ by integrating a vector field tangent to level spheres.
The result is shown in Figure~\ref{knot_block_comp}.
\begin{figure}[h!]
    \centerline{\includegraphics[scale=1.0]{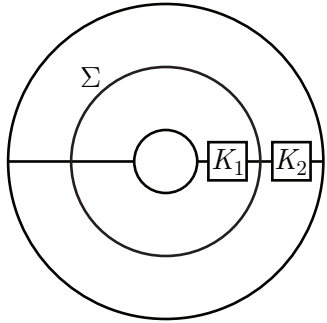}}
    \caption{Block $B(k)$ after ambient isotopy carrying $\Sigma$ to a level sphere.}
\label{knot_block_comp}
\end{figure}
As $k$ is prime, one of the diagrams $K_1$ or $K_2$ must be trivial.
\end{proof}

\begin{corollary}\label{inf_ib}
There exists a countably infinite collection of irreducible knot blocks, pairwise distinct up to diffeomorphism.
\end{corollary}

\begin{proof}
Let $\mathcal{T}$ denote the set of torus knots $\mathfrak{t}(a,b)$ where $a>b\geq 2$ and $\gcd(a,b)=1$.
As torus knots are prime~\cite[p.~95]{bz}, Lemma~\ref{pk_ib} implies that each $B(k)$, $k\in\mathcal{T}$, is irreducible.
The fundamental groups of these torus knots are pairwise nonisomorphic~\cite[p.~47]{bz}.
By Example~\ref{knot_blocks}, these knot blocks are pairwise distinct up to diffeomorphism.
\end{proof}

\begin{remark}
Under concatenation, knot blocks commute, unlike distinct blocks in $\B$ (recall Table~\ref{magma_data_2}).
Concatenating infinitely many knot blocks yields a multiray in $\R^3$ known as \emph{Wilder rays}.
They were classified by Fox and Harrold~\cite{foxharrold}.
\end{remark}

\section{Unknotted Ball-Arc Pairs}\label{ball_arc_pairs}

This section identifies some unknotted ball-arc pairs in blocks.
These tools will be used in the next two sections.
Recall that a \textbf{ball-arc pair} is a pair $\pa{\Delta,a}$ such that $a\approx D^1$ is neatly embedded in $\Delta\approx D^3$.
Such a pair is \textbf{unknotted} provided it is diffeomorphic to the standard pair $\pa{D^3, \cpa{(0,0)}\times D^1}$, and otherwise it is \textbf{knotted}.
If $k\subset S^3$ is a smooth knot (not the unknot) and $\pa{D,b}\subset\pa{S^3,k}$ is an unknotted ball-arc pair such that $D\cap k=b$, then
$\pa{S^3 -\Int{D}, k-\Int{b}}$ is a knotted ball-arc pair. Every knotted ball-arc pair arises this way up to diffeomorphism.\\

\begin{lemma}\label{ball-arc}
Let $\varepsilon_2=\pa{S^2\br{1,2},\tau}$.
Let $\Sigma$ be a $2$-sphere embedded in the interior of $S^2\br{1,2}$ and transverse to $\tau$.
Assume $\Sigma$ meets $\tau$ at exactly two points $p$ and $q$, both of which lie on one component of $\tau$, say $\tau_1$.
Then, $\Sigma$ bounds a $3$-disk, $\Delta$, in $S^2\br{1,2}$ and $\pa{\Delta,\Delta\cap\tau}$ is an unknotted ball-arc pair.
\end{lemma}

\begin{proof}
By Lemma~\ref{ess_sphere}, $\Sigma$ is inessential in $S^2\br{1,2}$.
Let $\Pi$ denote the annulus where $S^2\br{1,2}$ meets the $xy$-plane.
Without loss of generality, $\tau\subset\Pi$ and $\Sigma$ is transverse to $\Pi$.
Thus, $\Sigma\cap\Pi$ is a closed $1$-manifold and one component, $K$, of $\Sigma\cap\Pi$ must contain $p$ and $q$;
this is where the hypothesis $\Sigma\cap\tau_2=\emptyset$ is used.
As in the proof of Proposition~\ref{sigma_en} (paragraph three), we may arrange that $\Sigma\cap\Pi=K$.
Now, it is straightforward to construct the required diffeomorphism (cf. paragraph four of the proof of Proposition~\ref{sigma_en}).
\end{proof}

\begin{remarks}
\noindent\begin{enumerate}[label=(\arabic*),leftmargin=*]\setcounter{enumi}{0}
\item Lemma~\ref{ball-arc} becomes false without the hypothesis $\Sigma\cap\tau_2=\emptyset$ (i.e., with $\varepsilon_2$ replaced by $\varepsilon_1$).
To see this, consider the block $B=\pa{S^2\br{1,2},\tau'}$ in Figure~\ref{sphere_two_points} (left).
\begin{figure}[h!]
    \centerline{\includegraphics[scale=1.0]{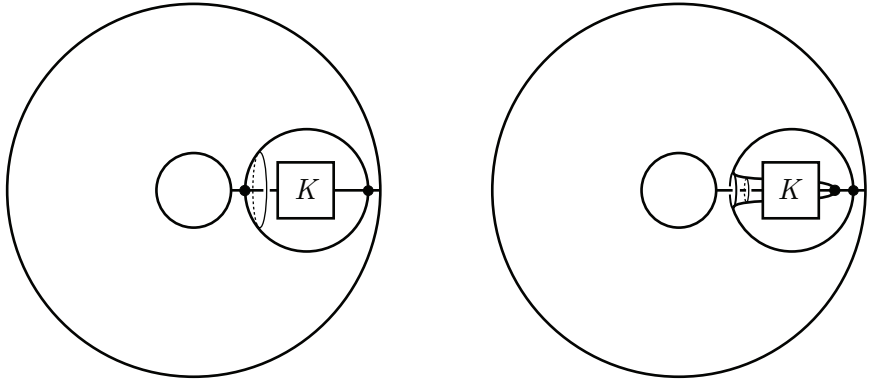}}
    \caption{Sphere $\Sigma'$ meeting the tangle $\tau'$ at two points in the block $B\approx\varepsilon_1$.
   						At right is the result of an ambient isotopy that fixes $\tau'$ setwise.}
\label{sphere_two_points}
\end{figure}
The indicated sphere $\Sigma'$ meets $\tau'$ in two points and bounds the $3$-ball $\Delta'$.
Let $K$ be any crossing diagram such that $\pa{\Delta',\Delta'\cap\tau'}$ is a knotted ball-arc pair.
Straighten $\tau'$ using the argument in Example~\ref{one_comp}.
Let $\tau$, $\Sigma$, and $\Delta$ denote the respective images of $\tau'$, $\Sigma'$, and $\Delta'$ under this ambient isotopy.
Then, $\Sigma$ is a $2$-sphere in $\varepsilon_1$ meeting $\tau$ in exactly two points and transversely.
However, $\pa{\Delta,\Delta\cap\tau}\approx \pa{\Delta',\Delta'\cap\tau'}$ is a knotted ball-arc pair.
\item Lemma~\ref{ball-arc} and the previous remark may be recast in $\R^3$ as follows.
Consider a $2$-sphere $\Sigma\subset\R^3$.
Let $\Delta\subset\R^3$ be the $3$-disk with $\partial\Delta=\Sigma$.
Suppose $a$ is a straight arc in $\R^3$ that is neatly embedded in $\Delta$.
Let $l\subset\R^3$ be the straight line containing $a$.
If $\Sigma$ is disjoint from $l-a$, then $\pa{\Delta,a}$ is an unknotted ball-arc pair.
If $\Sigma$ meets $l-a$, then $\pa{\Delta,a}$ may be a knotted ball-arc pair.
In fact, every knotted ball-arc pair $\pa{\Delta',a'}$ in $\R^3$ is ambient isotopic to some such $\pa{\Delta,a}$.
Proof: (i) straighten $a'$ near an endpoint $q'$, (ii) let $q''\neq q'$ be a point in the straightened end of $a'$, and (iii) ambiently isotop the other endpoint of $a'$ along $a'$ until it concides with $q''$ (cf. Figure~\ref{sphere_two_points} (right)). $\square$
\end{enumerate}
\end{remarks}

\begin{lemma}\label{disk_ess_sphere}
Let $a$ be a neatly embedded arc in $S^2\br{t_1,t_2}$.
Assume $a$ has one boundary point in $S^2\br{t_1}$ and the other in $S^2\br{t_2}$.
Let $D_1$ be a $2$-disk embedded in the interior of $S^2\br{t_1,t_2}$.
Assume that $C:=\partial D_1$ lies in some $S^2\br{t}$, $C$ is disjoint from $a$, $\Int{D_1}$ is disjoint from $S^2\br{t}$, and $D_1$ is transverse to $S^2\br{t}$.
Let $D_2$ and $D'_2$ be the two $2$-disks in $S^2\br{t}$ bounded by $C$.
Then:
\begin{enumerate}\setcounter{enumi}{\value{equation}}
\item\label{indneq} The intersection numbers $\#_2(a,D_2)$ and $\#_2(a,D'_2)$ are unequal.
\item The intersection numbers $\#_2(a,D_1)$ and $\#_2(a,D_2)$ are equal (after possibly interchanging the names of $D_2$ and $D'_2$).
\item The sphere $D_1\cup D_2$ is inessential in $S^2\br{t_1,t_2}$.
\item The sphere $D_1\cup D'_2$ is essential in $S^2\br{t_1,t_2}$.
\setcounter{equation}{\value{enumi}}
\end{enumerate}
\end{lemma}

\begin{proof}
Immediate by Lemma~\ref{ess_sphere}.
\end{proof}

Let $\R^n_+ :=\cpa{x\in\R^n\mid x_n\geq 0}$ denote closed upper half space.
The closed upper half disk is $D^n_+ := D^n \cap \R^n_+$.

\begin{lemma}\label{ball-arc_diffeo}
Let $\varepsilon_2=\pa{S^2\br{1,2},\tau}$.
Let $\tau_1$ be a component of $\tau$ and let $q=\tau_1\cap S^2\br{2}$.
Suppose $D_1$ is a $2$-disk neatly embedded in $S^2\br{1,2}$ such that: (i) $D_1$ is transverse to $\tau$, (ii) $D_1$ meets $\tau$ at one point $p\in\Int{\tau_1}$,
and (iii) $C=\partial D_1$ lies in $S^2\br{2}$.
Let $D_2$ be the $2$-disk in $S^2\br{2}$ with boundary $C$ and containing $q$.
Then, $D_1 \cup D_2$ bounds a piecewise smooth $3$-disk $D\subset S^2\br{1,2}$ and $D\cap\tau_2=\emptyset$.
Further, there is a diffeomorphism of pairs $g:\pa{D,D\cap\tau}\to \pa{D^3_+,\cpa{\pa{0,0}}\times D^1_+}$ that sends $D_1$ to the upper hemisphere and $D_2$ to $D^2\times\cpa{0}$.
\end{lemma}

\begin{proof}
By Lemma~\ref{disk_ess_sphere} with $a=\tau_1$, $D_1 \cup D_2$ is inessential in $S^2\br{1,2}$.
By hypothesis, $D_1\cap\tau_2=\emptyset$.
So, $D_2\cap\tau_2=\emptyset$, and $D\cap\tau_2=\emptyset$ as well.
The required diffeomorphism $g$ is constructed in bootstrapping fashion (cf. paragraph four of the proof of Proposition~\ref{sigma_en}):
define $g$ on $\partial D$, extend to a smooth, regular neigborhood of $\partial D$ in $D$, and extend to the rest of $D$ utilizing Lemma~\ref{ball-arc}.
\end{proof}

\section{Irreducibility of Blocks in \texorpdfstring{$\B$}{B}}\label{AB_irred}

\begin{theorem}\label{A_irred}
Each block in $\B$ is irreducible.
\end{theorem}

The remainder of this section is devoted to proving $A$ is irreducible, which suffices to prove Theorem~\ref{A_irred}.
Recall the block $A=\pa{S^2\br{1,2},T}$ from Figure~\ref{four_blocks}.
The (general position) immersion $f$ yielding $A$ (see Figure~\ref{immersion}) plays a central role in our proof.
We pause to explain $f$ and fix some notation.
We assume the reader has Figure~\ref{immersion} at hand.\\

The domain of $f$ is $\Omega:=S^1\br{1,2}\subset\R^2$.
The compact annulus $\Omega$ contains three equally spaced radial arcs, $\tau_1$, $\tau_2$, and $\tau_3$, as in Figure~\ref{immersion_domain}.
\begin{figure}[h!]
    \centerline{\includegraphics[scale=1.0]{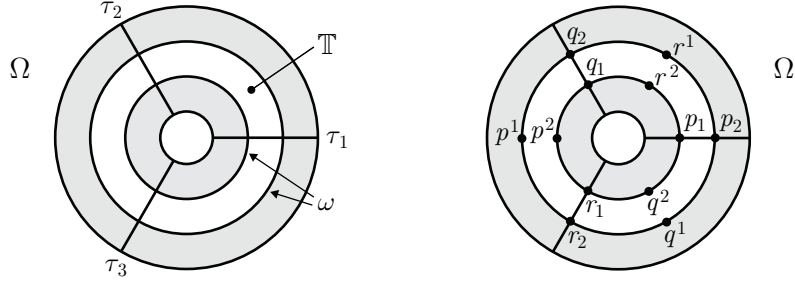}}
    \caption{Two labelings of the domain $\Omega$ of the immersion $f$.}
\label{immersion_domain}
\end{figure}
Let $\tau=\tau_1 \cup \tau_2 \cup \tau_3$.
Note that $f(\tau)=T$, the tangle in $A$.
We let $\T:=S^1\br{\sfrac{4}{3},\sfrac{5}{3}}$, the subannulus of $\Omega$ that is unshaded in Figure~\ref{immersion_domain}.
The boundary of $\T$ is $\omega=\partial \T$, the disjoint union of $S^1\br{\sfrac{4}{3}}$ and $S^1\br{\sfrac{5}{3}}$.\\

Given a subset $X\subset\Omega$, it will be convenient to let $X'$ denote $f(X)$.
(A notable exception is $T=f(\tau)$.)
In particular, $\Omega'=f(\Omega)$, $\T'=f(\T)$, and $\omega'=f(\omega)$.\\

For each $t\in\br{1,2}$, $\left.f\right|S^1\br{t}$ is an embedding, namely the composition of: a rigid rotation, a homothety, and a translation in the $z$-direction.
We will see that:
\begin{enumerate}[label=(\arabic*),leftmargin=*]\setcounter{enumi}{0}
\item The multiple points of $f$ are double points where $f\pa{S^1\br{\sfrac{4}{3}}}=f\pa{S^1\br{\sfrac{5}{3}}}=\omega'$.
\item $\omega'=S^1\br{\sfrac{5}{3}}\times\cpa{\varepsilon}$ where $0<\varepsilon<<1$.
\item $\T'$ is a torus in $S^2\br{1,2}$, smooth except for corners along $\omega'$.
\end{enumerate}
On $\partial \Omega$, $f$ is inclusion $\iota:(x,y)\mapsto(x,y,0)$.
\begin{figure}[h!]
    \centerline{\includegraphics[scale=1.0]{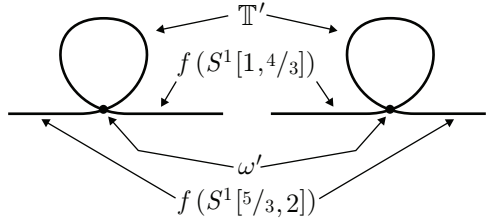}}
    \caption{Intersection of $\Omega'=f\pa{\Omega}$ with plane containing $z$-axis.}
\label{slice}
\end{figure}
Figure~\ref{slice} shows the intersection of $\Omega'$ with any plane in $\R^3$ containing the $z$-axis.
On each of the three subannuli of $\Omega$ in Figure~\ref{immersion_domain}, $f$ is defined as follows.
The annulus $S^1\br{1,\sfrac{4}{3}}$ is stretched radially to $S^1\br{1,\sfrac{5}{3}}$, then the outer boundary component is twisted by $\pi$ radians CCW
while fixing the inner boundary component, then level circles near the outer boundary component are lifted up a bit in the $z$-direction (to yield general position).
The annulus $S^1\br{\sfrac{5}{3},2}$ maps into $S^2\br{1,2}$ by $\iota$, then level circles near the inner boundary component are lifted up a bit in the $z$-direction
(again, to yield general position).
Finally, $f$ is defined on $\T$, interpolating $\left.f\right|S^1\br{1,\sfrac{4}{3}}$ and $\left.f\right|S^1\br{\sfrac{5}{3},2}$, so as to yield a torus $\T'=f\pa{\T}$ as in Figures~\ref{immersion} and~\ref{slice}.
The two components of $\omega$ are identified under $f$ after half a rotation of $S^1\br{\sfrac{4}{3}}$.
This completes our description of $f$.\\

For distinct $i,j\in\cpa{1,2,3}$, let $S_{i,j}$ denote the closed sector in $\Omega$ between $\tau_i$ and $\tau_j$ of angular measure $2\pi/3$.
Note that $\left.f\right|S_{i,j}$ is an embedding.
Fix distinct $i,j,k\in\cpa{1,2,3}$.
Observe that $T_k=f\pa{\tau_k}$ meets $f\pa{S_{i,j}}$ (transversely) at exactly two points.
For example, using the labelings in Figure~\ref{immersion_domain}, $T_1$ meets $f\pa{S_{2,3}}$ at the two points:
\[
	f(p_1)=f(p^1) \quad \tn{and} \quad f(p_2)=f(p^2)
\]
Similarly, $f(q_i)=f(q^i)$ and $f(r_i)=f(r^i)$ for $i=1$ and $2$.
The points $p^i$, $q^i$, and $r^i$, where $i=1$ and $2$, will be referred to as \textbf{special points}.\\

We prove $A$ is irreducible according to Definition~\ref{irred_def_1}.
It suffices to consider a $2$-sphere, $\Sigma$,
embedded in $\Int S^2\br{1,2}$, transverse to $T$, and meeting each component of $T$ at exactly one point.
We improve $\Sigma$ by ambient isotopies of $S^2\br{1,2}$ that fix $T$ setwise at all times.
By an abuse, we refer to each improved $\Sigma$ as $\Sigma$.
We view $\Omega'=f(\Omega)$ as an auxiliary object, unaffected by these isotopies.
Perturb $\Sigma$ so that $\Sigma\cap T$ is disjoint from $\omega'$.
Perturb $\Sigma$ again so that further $\Sigma$ meets $\Omega'$ in general position.
In particular, $\Sigma\cap\Omega'$ is an immersed, closed $1$-manifold in $\Sigma$ in general position.
Define:
\[
	\sigma:=f^{-1}\pa{\Sigma}=f^{-1}\pa{\Sigma\cap\Omega'}\subset\Omega
\]
which is an embedded, closed $1$-manifold in $\Int \Omega$, transverse to $\omega$ and $\tau$.
Each component, $\tau_i$, of $\tau$ meets $\sigma$ at exactly one point (not in $\omega$).
So, there exists one component, $K$, of $\sigma$ that meets each $\tau_i$ at one point (transversely) and $K$ is essential in $\Omega$.

\begin{claim}\label{suff_claim}
It suffices to arrange that $\sigma=K$ and $K\cap\T=\emptyset$.
\end{claim}

\begin{proof}
Similar to the argument in paragraph four of the proof of Proposition~\ref{sigma_en}.
\end{proof}

We give three operations for improving $\Sigma$.
Define the \textbf{complexity} of $\Sigma$ to be:
\[
c(\Sigma):=\card{\omega\cap\sigma}+(\tn{$\#$ of components of $\sigma$})\in\Z^{+}
\]

\begin{figure}[h!]
\centerline{\includegraphics[scale=1.0]{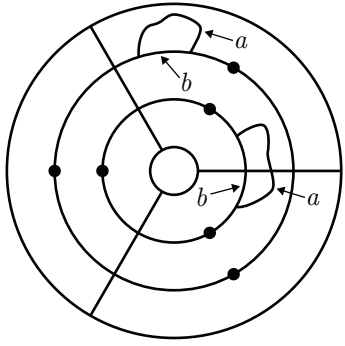}}
\caption{Two pairs of arcs satisfying the hypotheses of Lemma~\ref{taco0}.}
\label{taco_hypo}
\end{figure}

\begin{lemma}\label{taco0}
Suppose that $a\subset\sigma$ and $b\subset\omega$ are arcs, and $a\cup b$ is a simple closed curve bounding a disk $D\subset\Omega$.
Assume that $D$ contains no special points, $D\cap\omega=b$, and $\Int D\cap \sigma=\emptyset$ (see Figure~\ref{taco_hypo}).
Then, the points of $\partial a$ can be eliminated from $\sigma\cap\omega$, and $c(\Sigma)$ decreases by at least $3$.
\end{lemma}

\begin{proof}
As $D$ intersects only one component of $\omega$, $\left.f\right|D$ is an embedding.
The disk $D'=f(D)$ permits construction of an isotopy of $\Sigma$, with support near $D'$, that carries $f(a)$ past $f(b)$ to a parallel copy of $f(b)$.
If $D'$ intersects $T$, then $\pa{D,D\cap\tau}$ is a disk-arc pair (all of which are unknotted).
So, $\pa{D',D'\cap T}$ is an unknotted disk-arc pair, and the isotopy fixes $T$ setwise.
The reduction in $c(\Sigma)$ follows from Figure~\ref{taco_complexity}.
\begin{figure}[h!]
\centerline{\includegraphics[scale=1.0]{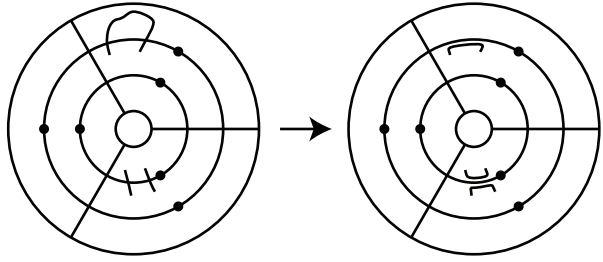}}
\caption{Instance of first operation. The points $\partial a$ of $\sigma\cap S^1\br{\sfrac{5}{3}}$ correspond to two points of $\sigma\cap S^1\br{\sfrac{4}{3}}$.}
\label{taco_complexity}
\end{figure}
Four points of $\sigma\cap\omega$ are eliminated, and, at worst, the number of components of $\sigma$ increases by one.
\end{proof}

\begin{lemma}\label{inessential}
Among the components of $\sigma$ that are inessential in $\Omega$ and disjoint from $\omega$, let $C$ be one that is innermost in $\Omega$.
Then, $C$ can be eliminated from $\sigma$, and $c(\Sigma)$ decreases by at least $1$.
\end{lemma}

\begin{proof}
Let $D\subset\Omega$ be the $2$-disk with $\partial D=C$.
Note that $D$ is disjoint from $\omega$ and $\tau$.
So, $\left.f\right|D$ is an embedding and $D'=f(D)$ is a $2$-disk disjoint from $T$ and bounding $C'=f(C)$.
The circle $C'$ bounds two $2$-disks, $D_1$ and $D_2$, in $\Sigma$.
The arc $T_1$ meets $\Sigma$ at one point.
So, without loss of generality, $T_1$ meets $D_1$ at one point (transversely, and in $\Int D_1$), and $T_1\cap D_2=\emptyset$.
By hypothesis, $\Int D \cap \sigma=\emptyset$.
So, $D' \cup D_1$ and $D'\cup D_2$ are embedded $2$-spheres.
By Lemma~\ref{ess_sphere}, $D'\cup D_1$ is essential in $S^2\br{1,2}$ and $D'\cup D_2$ is inessential.
So, $D'\cup D_2$ bounds an embedded $3$-disk $\Delta\subset S^2\br{1,2}$, and $\Delta\cap T=\emptyset$.
The $3$-disk $\Delta$ permits construction of an isotopy of $\Sigma$, with support near $\Delta$, that carries $D_2$ past $D'$ to a parallel copy of $D'$.
\end{proof}

\begin{lemma}\label{one_special_point}
Let $C$ be a component of $\sigma$ that bounds a $2$-disk $D\subset\Omega$.
Assume that $z:=D\cap\omega$ is a neatly embedded arc in $D$, $D$ contains exactly one special point $x$, $x\in\Int z$, and $\Int D\cap\sigma=\emptyset$.
Then, $C$ can be eliminated from $\sigma$, and $c(\Sigma)$ decreases by at least $5$.
\end{lemma}

\begin{proof}
Without loss of generality, $x=p^1$ as in Figure~\ref{ball-arc-push}.
\begin{figure}[h!]
\centerline{\includegraphics[scale=1.0]{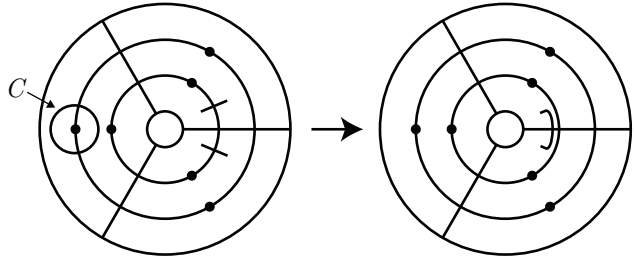}}
\caption{Instance of third operation. The points $C\cap \omega$ correspond to two points of $\sigma\cap S^1\br{\sfrac{4}{3}}$.}
\label{ball-arc-push}
\end{figure}
The embedded $2$-disk $D'=f(D)$ meets $T$ at exactly $f(p^1)=f(p_1)\in T_1$.
The circle $C'=f(C)$ bounds two $2$-disks, $D_1$ and $D_2$, in $\Sigma$.
Without loss of generality, $T_1\cap D_1=\emptyset$ and $T_1$ meets $D_2$ at one point (transversely).
By Lemma~\ref{ess_sphere}, $D'\cup D_2$ bounds an embedded $3$-disk $\Delta\subset \Int S^2\br{1,2}$ and $\Delta\cap T\subset T_1$ is a neatly embedded arc in $\Delta$.
By Lemma~\ref{ball-arc}, $(\Delta,\Delta\cap T)$ is an unknotted ball-arc pair.
The pair $(\Delta,\Delta\cap T)$ permits construction of an isotopy of $\Sigma$, with support near $\Delta$, that carries $D_2$ past $D'$ to a parallel copy of $D'$.
This isotopy fixes $T_1$ setwise and fixes $T_2$ and $T_3$ pointwise.
\end{proof}

Improve $\Sigma$ by applying Lemmas~\ref{taco0}, \ref{inessential}, and \ref{one_special_point} \emph{in any order and as long as possible}.
This is a finite process since the initial complexity of $\Sigma$ is a positive integer and each operation strictly reduces the complexity.
The complexity of the resulting improved $\Sigma$ is $c(\Sigma)\in\Z^+$.
The rest of this section shows that $c(\Sigma)=1$ and $K\cap\T=\emptyset$, which suffices to prove Theorem~\ref{A_irred} by Claim~\ref{suff_claim}.

\begin{lemma}\label{taco}
There do not exist arcs $a\subset\sigma$ and $b\subset\omega$ such that $a\cup b$ is a simple closed curve bounding a disk $D\subset\Omega$
where $D$ is disjoint from the special points.
\end{lemma}

\begin{proof}
We have $D\cap\omega=b\sqcup B$ where $B$ is a finite disjoint union of neatly embedded arcs in $D$.
We claim that $\Int D$ contains no closed component of $\sigma$.
Otherwise, let $\Delta\subset\Int D$ be the $2$-disk bounded by an innermost such component $C$.
If $C\cap\omega=\emptyset$, then Lemma~\ref{inessential} applies to $C$, a contradiction.
Thus, $\Delta\cap\omega$ is a nonempty finite disjoint union of neatly embedded arcs in $\Delta$.
Let $b_0$ be an arc of $\Delta\cap\omega$ that is outermost in $\Delta$ in the sense that $\partial b_0$ subtends an arc $a_0\subset\partial \Delta$
such that:
\[
a_0\cap \omega =\partial a_0 =\partial b_0
\]
Lemma~\ref{taco0} applies to $a_0$ and $b_0$, a contradiction.
The proof of the claim is complete.\\

If $B=\emptyset$, then Lemma~\ref{taco0} applies to $a$ and $b$, a contradiction.
Otherwise, there exists a component $b_0$ of $B$ that is outermost in $D$ in the sense that $\partial b_0$ subtends an arc $a_0\subset\partial D$ such that:
\[
a_0\cap B=\partial a_0 =\partial b_0
\]
Lemma~\ref{taco0} applies to $a_0$ and $b_0$, a contradiction.
\end{proof}

Next, we show that $\sigma$ contains no component inessential in $\Omega$.
Suppose, by way of contradiction, that $\sigma$ contains component(s) inessential in $\Omega$.
Among these components, there must be one, call it $C$, that is innermost in $\Omega$.
By Lemma~\ref{inessential}, $C$ meets $\omega$.
So, $\left|C\cap\omega\right|$ is positive and even.
Let $D\subset\Omega$ be the $2$-disk with $\partial D=C$.
Note that $\sigma\cap \Int D=\emptyset$.
Also, $B:=D\cap\omega$ is a nonempty, finite disjoint union of neatly embedded arcs in $D$.
Recall that the only component of $\sigma$ that meets $\tau$ is $K$, and $K$ is essential in $\Omega$.
So, $D$ is contained in the interior of a sector $S_{i,j}$ and $\left.f\right|D$ is an embedding.
Without loss of generality, assume $D\subset \Int S_{2,3}$.
Let $D'=f(D)$, an embedded $2$-disk with $\partial D'=C'=f(C)$.
The circle $C'$ also bounds two $2$-disks, $D_1$ and $D_2$, in $\Sigma$.
As $\sigma\cap\Int D=\emptyset$, $D'\cup D_1$ and $D'\cup D_2$ are embedded $2$-spheres in $S^2\br{1,2}$.
By Lemma~\ref{ess_sphere} (using the arc $T_1$, say), one of these spheres is essential in $S^2\br{1,2}$ and the other is inessential.
Without loss of generality, assume $D'\cup D_1$ is essential in $S^2\br{1,2}$ and $D'\cup D_2$ is inessential.
Let $\Delta$ be the $3$-disk in $S^2\br{1,2}$ with $\partial\Delta=D'\cup D_2$.

\begin{lemma}\label{no_inessential}
The disk $D$ cannot be disjoint from the special points $p^1$ and $p^2$.
\end{lemma}

\begin{proof}
Otherwise, let $b$ be a component of $B$.
Let $a\subset C$ be an arc with $\partial a=\partial b$.
The arcs $a$ and $b$ contradict Lemma~\ref{taco}.
\end{proof}

\begin{lemma}\label{2special}
The disk $D$ cannot contain both special points $p^1$ and $p^2$.
\end{lemma}

\begin{proof}
Suppose otherwise.
Note that $T_1$ meets $D'$ twice (transversely in $\Int D'$), and $T_2$ and $T_3$ are disjoint from $D'$.
By Lemma~\ref{ess_sphere}, $T$ is disjoint from $D_2$.
The disk $\Delta$ permits construction of an an ambient isotopy $F_t$, $0\leq t\leq 1$, of $S^2\br{1,2}$ that carries $D_2$ past $D'$ to a parallel copy of $D'$.
This isotopy has support near $\Delta$, is relative to a neighborhood of $T_2\cup T_3\cup\partial S^2\br{1,2}$, but does \emph{not} fix $T_1$ setwise.
Note that $F_1(T_1)\cap S'_{2,3}=\emptyset$ where $S'_{2,3}$ still denotes $f(S_{2,3})$.
As $F_1(A)\approx A$ is a Borromean block, there is an ambient isotopy $G_t$, $0\leq t\leq1$, of $S^2\br{1,2}$, relative to a neighborhood of 
$\partial S^2\br{1,2}$, that straightens both $F_1(T_1)$ and $F_1(T_2)=T_2$.
The strip $G_1(S'_{2,3})$ permits construction of an ambient isotopy $H_t$, $0\leq t\leq 1$, of $S^2\br{1,2}$ that carries $G_1(F_1(T_3))=G_1(T_3)$ to an
arc close to and winding around $G_1(F_1(T_2))=G_1(T_2)$.
This isotopy has support near $G_1(S'_{2,3})$, is relative to both $G_1(F_1(T_1))$ and $G_1(F_1(T_2))=G_1(T_2)$, but otherwise is not relative to $\partial S^2\br{1,2}$.
The resulting block $B=H_1 G_1 F_1 (A)$ is diffeomorphic to $\varepsilon_3$ by untwisting tangle components $2$ and $3$.
So, $A\approx B \approx \varepsilon_3$, which contradicts Corollary~\ref{Anottrivial}.
\end{proof}

\begin{lemma}\label{one}
The disk $D$ cannot contain exactly one of the special points $p^1$ or $p^2$.
\end{lemma}

\begin{proof}
Assume, without loss of generality, that $\Int D$ contains $p^1$ but not $p^2$.
Recall that $\sigma\cap \Int D=\emptyset$.
If $B$ is connected, then Lemma~\ref{one_special_point} applies to $D$, a contradiction.
If $B$ is disconnected, then $B$ contains a component, $b$, that does not contain $p^1$.
Let $a\subset C$ be the arc with $\partial a=\partial b$ and such that $p^1$ does not lie inside the simple closed curve $a\cup b$.
The arcs $a$ and $b$ contradict Lemma~\ref{taco}.
\end{proof}

Taking stock, $\sigma$ contains no component inessential in $\Omega$.
So, $\sigma=K$ is a single, essential circle in $\Omega$, and $K$ meets each $\tau_i$ at one point (transversely).
It remains to prove that $K\cap\T=\emptyset$.
A \textbf{segment} will mean an arc $s \subset K$ for which $s \cap (\omega \cup \tau) = \partial s$.
Each inline figure, such as \includegraphics[scale=0.6]{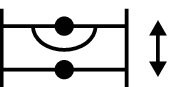}, represents a segment of $K$ in a sector;
vertical lines represent adjacent components of $\tau$, upper and lower horizontal lines represent arcs of $\omega$,
dots represent special points, and arrows indicate reflected cases of the entire figure.
For example, \includegraphics[scale=0.6]{scin_arrow.eps} represents six cases (three choices of sector and a possible vertical reflection), and \includegraphics[scale=0.6]{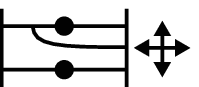} represents twelve cases.

\begin{lemma}\label{scin}
None of the following appear: \includegraphics[scale=0.6]{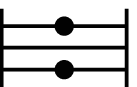}, \includegraphics[scale=0.6]{scin_arrow.eps}, or \includegraphics[scale=0.6]{exitm_arrow.eps}.
\end{lemma}

\begin{proof}
Observe that points of $\Omega$ inside (outside) $K$ map under $f$ to lie inside (outside) $\Sigma$ respectively.
Suppose there is \includegraphics[scale=0.6]{middle.eps}. Without loss of generality, assume the indicated special points are $p^1$ and $p^2$.
Then, $f(p^1)$ is outside $\Sigma$ and $f(p^2)$ is inside $\Sigma$.
There are three possibilities for the location of $K\cap\tau_1$ (see Figure~\ref{annulus_omo} at right).
The inner option implies $f(p_1)$ and $f(p_2)$ both lie outside $\Sigma$.
The middle option implies $f(p_1)$ lies inside $\Sigma$ and $f(p_2)$ lies outside $\Sigma$.
The outer option implies $f(p_1)$ and $f(p_2)$ both lie inside $\Sigma$.
All three are contradictions since $f(p_1)=f(p^1)$ and $f(p_2)=f(p^2)$.
So, no \includegraphics[scale=0.6]{middle.eps} appears.\\

We claim that:
\begin{enumerate}\setcounter{enumi}{\value{equation}}
\item\label{not_both} There do not exist \includegraphics[scale=0.6]{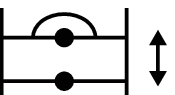} and \includegraphics[scale=0.6]{scin_arrow.eps} in the same sector and with the same reflection.
\setcounter{equation}{\value{enumi}}
\end{enumerate}
To see this, suppose, by way of contradiction, that \includegraphics[scale=0.6]{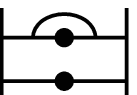} and \includegraphics[scale=0.6]{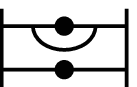} both appear in the same sector,
say $S_{2,3}$.
Let $z_+$ and $z_-$ denote these segments respectively.
Assume, without loss of generality, that the special point just below $z_+$ is $p^1$.
Let $z\subset\omega$ be the unique arc such that $p^1\in\Int z$ and $z\cap\sigma=\partial z$.
Lemma~\ref{taco} implies that $\partial z$ bounds segments $s_+$ parallel to $z_+$ and $s_-$ parallel to $z_-$ ($s_+=z_+$ and $s_-=z_-$ are possible).
Then, $s_+ \cup s_-$ is an inessential component of $\sigma$, a contradiction.
This completes our proof of~\ref{not_both}.\\

Suppose there is \includegraphics[scale=0.6]{scin.eps}.
Assume, without loss of generality, that this segment $s$ of $K$ lies in $S_{2,3}$.
By Lemma~\ref{taco}, the only possible segment of $K$ that meets both $\T$ and $\Omega-\T$,
and is disjoint from $\tau$, is \includegraphics[scale=0.6]{scout_arrow.eps}.
By~\ref{not_both}, the boundary points of $s$ lie in distinct segments, $s_2$ and $s_3$, of $K$, where $s_2$ meets $\tau_2$ at one endpoint and $s_3$ meets $\tau_3$ at one endpoint (see Figure~\ref{annulus_omo}).
\begin{figure}[h!]
    \centerline{\includegraphics[scale=1.0]{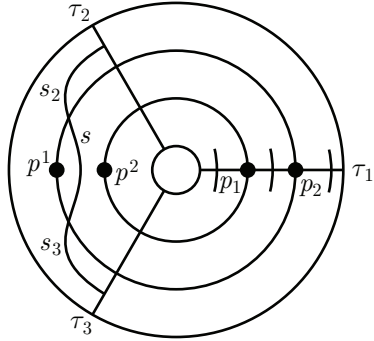}}
    \caption{Hypothetical arrangement of segments of $K$.}
    \label{annulus_omo}
\end{figure}
So, $f(p^1)$ is outside $\Sigma$ and $f(p^2)$ is inside $\Sigma$.
This yields the same contradiction as for \includegraphics[scale=0.6]{middle.eps} above.
Hence, no \includegraphics[scale=0.6]{scin_arrow.eps} appears.\\

Suppose there is \includegraphics[scale=0.6]{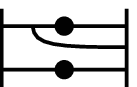}.
Call this segment $s$ and assume, without loss of generality, that $p^1$ is the special point pictured above $s$.
Let $z\subset\omega$ denote the short arc with $\partial z =\cpa{q_2,p^1}$ (see Figure~\ref{annulus_omm}).
We claim that no segment of $K$ contained in $\T$ may meet $z$.
This follows from: (i) Lemma~\ref{taco}, (ii) the nonexistence of \includegraphics[scale=0.6]{scin_arrow.eps},
(iii) since $s$ meets $\tau_2$, and (iv) since $K$ meets $\tau_2$ exactly once.
Therefore, $K\cap S_{2,3}$ appears as in Figure~\ref{annulus_omm}.
\begin{figure}[h!]
    \centerline{\includegraphics[scale=1.0]{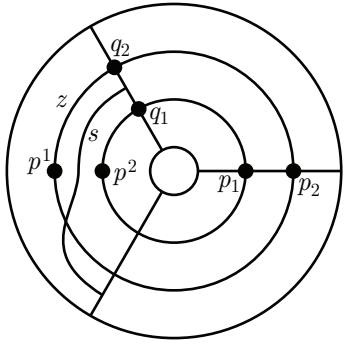}}
    \caption{$K$ cannot exit the annulus $\T$ between $q_2$ and $p^1$.}
    \label{annulus_omm}
\end{figure}
So, $f(p^1)$ is outside $\Sigma$ and $f(p^2)$ is inside $\Sigma$, a familiar contradiction.
Hence, no \includegraphics[scale=0.6]{exitm_arrow.eps} appears.
\end{proof}

Recall that $\T'\subset S^2\br{1,2}$ is an embedded torus, smooth except for corners along $\omega'$.
So, $\Sigma\cap\T'$ is a closed $1$-manifold, smooth except for corners, and embedded in $\Sigma$ and in $\T'$.
We introduce a based longitude $\lambda$ and a based meridian $\mu$ on $\T'$.
Both originate and terminate at $f(p_2)\in\omega'$ (recall Figures~\ref{immersion}, \ref{immersion_domain}, and \ref{slice}).
The longitude $\lambda$ runs once along $\omega'$ in the CW direction about the $z$-axis when viewed from the point $(0,0,1)$.
The meridian $\mu$ is the right loop in Figure~\ref{slice}, oriented CCW, where the plane of intersection is the $xz$-plane.
In particular, a parallel pushoff $\lambda_0$ of $\lambda$ into the inside of $\T'$ has linking number $+1$ with $\mu$ in $\R^3$.
An oriented loop in $\T'$ has \textbf{type} $(m,n)$ provided it is freely homotopic in $\T'$ to $\mu^m \lambda^n$. 

\begin{lemma}\label{essential}
Let $C$ be a component of $\Sigma\cap\T'$, equipped with an orientation and of type $(m,n)$.
Then, $mn=0$.
\end{lemma}

\begin{proof}
Assume $C$ is essential in $\T'$ (otherwise, the result is clear).
Note that $\gcd\pa{m,n}=1$.
Focus attention on the submanifolds $\Sigma$ and $\T'$ of $S^2\br{1,2}\subset\R^3$.
We view $\Sigma$ and $\T'$ as submanifolds of $S^3=\R^3\cup\cpa{\infty}$, where $S^3=(S^1\times D^2)\cup (D^2\times S^1)$ and $(S^1\times D^2)\cap (D^2\times S^1)=\T'$.
As $C$ is essential in $\T'$, $C$ is not null-homologous (denoted $C\not\sim 0$) in both $S^1\times D^2$ and $D^2\times S^1$.
Let $X$ denote $S^1\times D^2$ or $D^2\times S^1$ where $C\not\sim 0$ in $X$.
Exactly one component of $\Sigma-\T'$ contains $C$ in its frontier and lies in $X$; let $\Gamma$ denote the closure of this component in $\Sigma$.
Note that $\Gamma$ is a compact $2$-disk with holes, $C\subset\partial \Gamma$, $\Gamma\subset X$, and $\partial \Gamma \subset\partial X=\T'$.
As $C\not\sim 0$ in $X$, there must be another component $C_0$ of $\partial \Gamma$ such that $C_0\not\sim0$ in $X$.
In particular, $C_0$ is essential in $\T'$. 
Choose an orientation of $C_0$.
Then, $C_0$ has type $(m_0,n_0)$ where $\gcd\pa{m_0,n_0}=1$.
The algebraic intersection number of $C$ and $C_0$ in $\T'$ equals $mn_0-nm_0$, which must vanish since $\Sigma\cap\T'$ is embedded in $\T'$.
It follows that $C_0$ has type $\pm(m,n)$.
Switching the orientation of $C_0$ if necessary, $C_0$ has type $(m,n)$.
Thus, $C$ and $C_0$ are parallel in $\T'$.
An exercise (left to the reader) shows that the linking number $\tn{lk}(C,C_0)$ in $\R^3$ equals $mn$.
As $C$ and $C_0$ are disjointly embedded in the sphere $\Sigma\subset\R^3$, $\tn{lk}(C,C_0)=0$.
Hence, $mn=0$ as desired.
\end{proof}

We define an \textbf{extended segment} to be an arc component of $K\cap\T$ of the form:
\[
\includegraphics[scale=0.6]{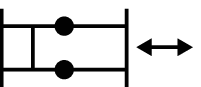}, \quad \includegraphics[scale=0.6]{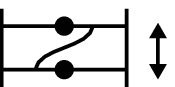}, \quad \tn{or}\quad \includegraphics[scale=0.6]{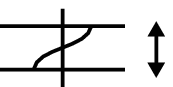} 
\]
In the latter type, the vertical line is a component of $\tau$ and the horizontal lines are arcs of $\omega$ not meeting special points.

\begin{corollary}\label{int_ext_seg}
Any component of $K\cap\T$ is an extended segment.
\end{corollary}

\begin{proof}
Let $s$ be a component of $K\cap\T$.
By Lemma~\ref{scin}, no \includegraphics[scale=0.6]{middle.eps} appears.
Therefore, $s\neq K$.
Hence, $s$ is an arc neatly embedded in $\T$.
Note that $\partial s\subset\omega$, and $\partial s$ is disjoint from $\tau$ and from special points.
By Lemma~\ref{taco}, no \includegraphics[scale=0.6]{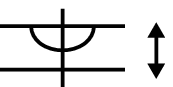} appears.
The result now follows from Lemma~\ref{scin} by considering the possible locations of points in $\partial s$.
\end{proof}

\begin{proposition}\label{disjoint}
$K\cap\T=\emptyset$.
\end{proposition}

\begin{proof}
Suppose, by way of contradiction, that $K$ meets $\T$.
Then, there is a circle $C\subset\Sigma\cap\T'$.
By Corollary~\ref{int_ext_seg}, there exists a finite, disjoint collection $s_1,s_2,\ldots,s_m$, $m\geq 1$, of extended segments
such that:
\[
	C=s'_1\cup s'_2 \cup \cdots \cup s'_m
\]
Orient all segments $s_k$, $1\leq k\leq m$, to point out from $S^2\br{\sfrac{4}{3}}$.
Note that this yields a coherent orientation of $C$.
A moment of reflection on the immersion $f$ (cf. Figure~\ref{immersion}) reveals that each $s'_k$ winds around $\T'$ by:
(i) $+1$ revolutions in the $\mu$ direction, and (ii) $\theta_k$ revolutions in the $\lambda$ direction where $1/6<\theta_k<5/6$.
Hence, $C$ has type $(m,n)$ where $m\geq1$ and $m/6<n<5m/6$, which contradicts Lemma~\ref{essential}.
\end{proof}

This completes our proof that $A$ is irreducible and our proof of Theorem~\ref{A_irred}. \qed

\section{Improving Spheres in Concatenations of Borromean Blocks}\label{improve_spheres}

This section proves that certain spheres in concatenations of Borromean blocks may be ambiently isotoped into a single block summand, while fixing the tangle setwise.

\begin{proposition}\label{borr_prime}
Let $B_i$, $1\leq i \leq k$, be Borromean blocks where $k\geq 2$.
Consider the concatenation:
\[
B:=B_1 B_2 \cdots B_k = \pa{S^2\br{1,k+1},\tau}
\]
Let $\Sigma$ be a $2$-sphere embedded in the interior of $S^2\br{1,k+1}$, transverse to $\tau$,
and meeting each component $\tau_i$ of $\tau$ at exactly one point $x_i$ for $i=1$, $2$, and $3$.
Then, there is a neighborhood $U$ of $\partial S^2\br{1,k+1}$ in $S^2\br{1,k+1}$ and an ambient isotopy $H_t$, $0\leq t \leq 1$, of $S^2\br{1,k+1}$ such that:
\begin{enumerate}\setcounter{enumi}{\value{equation}}
\item $H_0=\tn{Id}$.
\item $\left.H_t\right|U=\tn{Id}$ for all $0\leq t\leq 1$.
\item $H_t$ fixes $\tau$ setwise for all $0\leq t \leq 1$.
\item $H_1\pa{\Sigma}\subset\Int{S^2\br{m,m+1}}$ for some $m\in\cpa{1,2,\ldots,k}$.
\setcounter{equation}{\value{enumi}}
\end{enumerate}
\end{proposition}

\begin{remark}
For the definition of \emph{Borromean block}, see Section~\ref{borr_blocks}.
\end{remark}

\begin{proof}
Without loss of generality, no point $x_i$ lies in a sphere $S^2\br{j}$ where $j\in\cpa{2,3,\ldots,k}$, and $\Sigma$ is transverse to these spheres.
So, $Z:=\Sigma\cap \pa{\cup_{j=2}^{k} S^2\br{j}}$ is a closed $1$-manifold.
It suffices to improve $\Sigma$, by appropriate ambient isotopies of $S^2\br{1,k+1}$, so that $Z$ becomes empty.
We will employ the following two operations.\\

\noindent \textbf{Operation 1.} Suppose $C$ is a component of $Z$ bounding a disk $D_1\subset\Sigma$ such that:
(i) $D_1$ is disjoint from $Z-C$ and (ii) $D_1$ is disjoint from $\tau$.
The circle $C$ lies in $S^2\br{j}$ for some $j\in\cpa{2,3,\ldots,k}$ and bounds two $2$-disks, $D_2$ and $D'_2$, in $S^2\br{j}$.
By Lemma~\ref{disk_ess_sphere} (with $a=\textnormal{any }\tau_i$), we have $0=\#_2(a,D_1)=\#_2(a,D_2)$ and $D_1\cup D_2$ is inessential in $S^2\br{1,k+1}$.
So, $D_2\cap\tau=\emptyset$ (since each $\tau_i$ meets $D_2$ at most once) and there is a $3$-disk $D$ in $S^2\br{1,k+1}$ bounded by $D_1\cup D_2$.
Hence, $\tau$ is disjoint from $D$ and $D$ permits construction of an ambient isotopy, with support near $D$,
that carries $D_1$ past $D_2$ to a parallel copy of $D_2$.
Thus, $C$ (at least) has been eliminated from $Z$.\\

\noindent \textbf{Operation 2.} Suppose $C$ is a component of $Z$ bounding a disk $D_1\subset\Sigma$ such that:
(i) $D_1$ is disjoint from $Z-C$ and (ii) $D_1$ intersects $\tau$ at one point.
For notational convenience, we assume $D_1\cap\tau=\cpa{x_1}$.
Note that $x_1\in\Int{D_1}$.
The circle $C$ lies in $S^2\br{j}$ for some $j\in\cpa{2,3,\ldots,k}$.
Let $D_2\subset S^2\br{j}$ be the $2$-disk whose boundary is $C$ and whose interior meets $\tau_1$ (necessarily at one point, call it $q$).
Now, $D_1$ lies in $S^2\br{j-1,j}$ or in $S^2\br{j,j+1}$.
Without loss of generality, assume the latter.
Consider the block:
\[
X:=\pa{S^2\br{j,j+1},r}
\]
where $r$ has components:
\[
 r_i:=\tau_i \cap S^2\br{j,j+1} \tn{ for $i=1$, $2$, and $3$.}
\]
Let $X'$ be the block obtained from $X$ by forgetting $r_3$.
Then, $X'\approx\varepsilon_2$ since $X\approx B_j$ is a Borromean block.
Hence, Lemma~\ref{ball-arc_diffeo} implies that $D_1\cup D_2$ bounds the $3$-disk $D\subset S^2\br{j,j+1}$,
that $D\cap r_2=\emptyset$, and there exists a diffeomorphism $g:\pa{D,D\cap (r_1\cup r_2)}\to \pa{D^3_+,\cpa{\pa{0,0}}\times D^1_+}$.
The first two of these consequences imply that $D\cap \tau_2=\emptyset$.
Forgetting $r_2$ instead of $r_3$, we get $D\cap\tau_3=\emptyset$.
Therefore, we have a diffeomorphism:
\[
g:\pa{D,D\cap \tau}\to \pa{D^3_+,\cpa{\pa{0,0}}\times D^1_+}
\]
This diffeomorphism permits construction of an ambient isotopy of $S^2\br{1,k+1}$ that:
(i) has support near $D$, (ii) fixes $\tau_2$ and $\tau_3$ pointwise, (iii) fixes $\tau_1$ setwise,
and (iv) carries $D_1$ past $D_2$ to a parallel copy of $D_2$.
Thus, $C$ (at least) has been eliminated from $Z$.\\

Observe that if $Z\neq\emptyset$, then Operation 1 or 2 is applicable.
Indeed, let $C_0$ be a component of $Z$.
One component, $W$, of $\Sigma-C_0$ contains at most one of the points $x_1$, $x_2$, or $x_3$.
If $W$ contains components of $Z$, then Operation 1 or 2 applies to any innermost component of $Z$ in $W$.
If $W$ contains no component of $Z$, then Operation 1 or 2 applies to $C_0$ itself.
So, by finitely many applications of Operations 1 and 2, we get $Z=\emptyset$ and the proof is complete.
\end{proof}

\begin{corollary}\label{fin_cor}
Let $B_i$, $1\leq i \leq m$, and $C_j$, $1\leq j\leq n$, be irreducible Borromean blocks.
Consider the concatenations:
\begin{align*}
	B:=&B_1 B_2 \cdots B_m =\pa{S^2\br{1,m+1},\tau}\\
	C:=&C_1 C_2 \cdots C_n =\pa{S^2\br{1,n+1},\sigma}
\end{align*}
Suppose there is a diffeomorphism $f:B\to C$ such that $f(S^2)=S^2$.
Then, $m=n$ and there is an isotopy of $f$, sending $\tau$ to $\sigma$ setwise at all times, to a diffeomorphism $g:B\to C$ such that $g(B_i)=C_i$ for each $1\leq i\leq m$.
\end{corollary}

\begin{proof}
Assume $m\leq n$ (otherwise, consider $f^{-1}$).
If $m=1$, then $n>1$ contradicts irreducibility of $B_1$ (recall that each $C_j\not\approx\varepsilon_3$ since each $C_j$ is a Borromean block).
So, $n=1$ and we are done.
Next, let $m>1$.
All isotopies of $f$ send $\tau$ to $\sigma$ setwise at all times.
All isotopies of $C$ are ambient and fix $\sigma$ setwise at all times.
By an abuse of notation, the corresponding improved $f$ will still be denoted by $f$.
By Proposition~\ref{borr_prime}, we can isotop $f$ such that $f\pa{S^2[2]}\subset\Int C_j$ for some $j$.
As $C_j$ is irreducible, we can isotop $f$ so that $f\pa{S^2[2]}$ is a level $2$-sphere in $C_j$ very close to one boundary sphere of $C_j$, namely: (i) $S^2[j]$, or (ii) $S^2[j+1]$.\newline
\noindent Case 1. $j=1$ and (i) occurs. Then, $B_1\approx\varepsilon_3$, a contradiction since $B_1$ is a Borromean block.\newline
\noindent Case 2. $j=1$ and (ii) occurs, or $j=2$ and (i) occurs. Then, we may further isotop $f$ so that $f\pa{S^2[2]}=S^2[2]$.\newline
\noindent Case 3. $j=2$ and (ii) occurs, or $j>2$. These cases contradict irreducibility of $B_1$.\newline
\noindent In any case, we have arranged that $f\pa{S^2[2]}=S^2[2]$. Repeat this process with $f\pa{S^2[3]}$ and so forth, and it must terminate with $m=n$.
\end{proof}

\begin{corollary}\label{workhorse}
Let $B_i$ and $C_i$, $i\in\Z^+$, be sequences of irreducible Borromean blocks.
Consider the two Borromean rays:
\begin{align*}
	\pa{\R^3,\tau}&:=D^3 B_1 B_2 B_3\cdots\\
	\pa{\R^3,\sigma}&:=D^3 C_1 C_2 C_3 \cdots
\end{align*}
Suppose there is a diffeomorphism $f:\pa{\R^3,\tau}\to\pa{\R^3,\sigma}$.
Then, there exist $n\in\Z$ and $N\in\Z^+$, and there is an isotopy of $f$, sending $\tau$ to $\sigma$ setwise at all times,
to a diffeomorphism $g:\pa{\R^3,\tau}\to\pa{\R^3,\sigma}$ such that $g(B_i)=C_{i+n}$ for each $i\geq N$.
\end{corollary}

\begin{proof}
All isotopies send $\tau$ to $\sigma$ setwise at all times, and are relative to $D^3$.
By compactness, there exists an integer $N\geq 2$ such that $f(S^2\br{N})$ is disjoint from $D^3 C_1$.
By Proposition~\ref{borr_prime}, we may isotop $f$ so that $f\pa{S^2[N]}\subset \Int C_i$ for some $i\geq 2$.
As $C_i$ is irreducible, we may further isotop $f$ so that $f\pa{S^2[N]}=S^2[M]$ for some integer $M\geq 2$.
Define $n:=M-N\in\Z$.
Use Proposition~\ref{borr_prime} and irreducibility repeatedly to get $f\pa{S^2[i]}=S^2[i+n]$ for each integer $i\geq N$.
\end{proof}

\section{Borromean Rays and Hyperplanes}\label{brbh}

\subsection{Borromean Rays}\label{borr_rays}

This section proves Theorem~\ref{borr_rays_thm}, the first of our main results.
If $\sigma$ is a multiray given by a concatenation of blocks:
\[
\pa{\R^3,\sigma}=D^3 C_1 C_2 C_3\cdots
\]
then the \textbf{mirror} of $\sigma$, denoted $\overline{\sigma}$, is defined by:
\[
\pa{\R^3,\overline{\sigma}}:=D^3 \overline{C_1} \: \overline{C_2} \: \overline{C_3}\cdots
\]

\begin{theorem}\label{borr_rays_thm}
Let $B_i$ and $C_i$, $i\in\Z^{+}$, be sequences of blocks in $\B$.
Consider the two Borromean rays:
\begin{align*}
	\pa{\R^3,\tau}&:=D^3 B_1 B_2 B_3\cdots\\
	\pa{\R^3,\sigma}&:=D^3 C_1 C_2 C_3 \cdots
\end{align*}
There exists a diffeomorphism $f:\pa{\R^3,\tau}\to\pa{\R^3,\sigma}$ preserving orientation of $\R^3$ if and only if
there exists $n\in\Z$ such that one of the following holds for all sufficiently large $i\in\Z^{+}$:
\begin{enumerate}[itemsep=3pt]\setcounter{enumi}{\value{equation}}
\item\label{BeqC} $B_i=C_{i+n}$ (i.e., $B_i$ and $C_i$ have identical tails).
\item\label{BeqCbarstar} $B_i=\overline{C_{i+n}}^{\ast}$ (i.e., $B_i$ and $\overline{C_i}^{\ast}$ have identical tails).
\setcounter{equation}{\value{enumi}}
\end{enumerate}
There exists a diffeomorphism $f:\pa{\R^3,\tau}\to\pa{\R^3,\sigma}$ reversing orientation of $\R^3$ if and only if
there exists $n\in\Z$ such that one of the following holds for all sufficiently large $i\in\Z^{+}$:
\begin{enumerate}[itemsep=3pt]\setcounter{enumi}{\value{equation}}
\item\label{BeqCbar} $B_i=\overline{C_{i+n}}$ (i.e., $B_i$ and $\overline{C_{i}}$ have identical tails).
\item\label{BeqCstar} $B_i={C_{i+n}}^{\ast}$ (i.e., $B_i$ and ${C_{i}}^{\ast}$ have identical tails).
\setcounter{equation}{\value{enumi}}
\end{enumerate}
\end{theorem}

\begin{proof}
If~\ref{BeqC} holds, then $f$ exists by Remarks~\ref{concat_remarks} item~\ref{tail_det_type}.
Assume~\ref{BeqCbarstar} holds.
Then:
\begin{align*}
D^3 B_1 B_2 B_3\cdots &\approx D^3 B_N B_{N+1} B_{N+2}\cdots\\
	&= D^3 \overline{C_{M}}^{\ast} \, \overline{C_{M+1}}^{\ast} \, \overline{C_{M+2}}^{\ast} \cdots\\
	&\approx D^3  \overline{C_{1}}^{\ast} \, \overline{C_{2}}^{\ast} \, \overline{C_{3}}^{\ast} \cdots\\
	&\approx D^3 C_1 C_2 C_3 \cdots
\end{align*}
The last diffeomorphism is $(x,y,z)\mapsto(x,-y,-z)$ (= rotation of $\R^3$ about $x$-axis), followed by a simple ambient isotopy in each block relative to boundary $2$-spheres (cf. Section~\ref{s:diffeo_blocks}).
The other two diffeomorphisms come from Remarks~\ref{concat_remarks} item~\ref{tail_det_type}.
The composition is the required $f$.\\

Assume~\ref{BeqCbar} holds. The first orientation preserving case above yields a diffeomorphism $g:\pa{\R^3,\tau}\to\pa{\R^3,\overline{\sigma}}$ that preserves orientation of $\R^3$.
Composing $g$ with $(x,y,z)\mapsto(x,y,-z)$ yields the required $f$.
Assume~\ref{BeqCstar} holds. Then:
\[
	B_i =\overline{\pa{\overline{C_{i+n}}}}^{\ast}
\]
for sufficiently large $i$. The second orientation preserving case yields a diffeomorphism $g:\pa{\R^3,\tau}\to\pa{\R^3,\overline{\sigma}}$ that preserves orientation of $\R^3$. Again, compose $g$ with $(x,y,z)\mapsto(x,y,-z)$ to obtain the required $f$.\\

For the forward implications, note that blocks in $\B$ are irreducible by Theorem~\ref{A_irred}.
By Corollary~\ref{workhorse}, there exist $n\in\Z$, $N\in\Z^+$, and a diffeomorphism
$g:\pa{\R^3,\tau}\to\pa{\R^3,\sigma}$, isotopic to $f$, such that $g(B_i)=C_{i+n}$ for each $i\geq N$.
Assume $f$ preserves orientation of $\R^3$.
Then, $g$ preserves orientation of $\R^3$.
So, each of the diffeomorphisms:
\[
\left. g\right|:B_i \to C_{i+n}, \quad i\geq N,
\]
preserves orientation, preserves boundary $2$-spheres setwise, and has the same tangle permutation $\pi\in\tn{Sym}(3)$ (see Section~\ref{s:diffeo_blocks}).
If $\pi\in\tn{A}_3$, then~\ref{BeqC} holds by Table~\ref{hom_block} in Section~\ref{s:diffeo_blocks}.
If $\pi\in(1,2)\tn{A}_3=C$, then~\ref{BeqCbarstar} holds by Table~\ref{hom_block}.\\

Finally, assume the given $f$ reverses orientation of $\R^3$.
The composition of $f$ with $(x,y,z)\mapsto(x,y,-z)$ is a diffeomorphism $\pa{\R^3,\tau}\to\pa{\R^3,\overline{\sigma}}$ preserving orientation of $\R^3$.
Now, apply the orientation preserving case.
\end{proof}

Let $\mathcal{S}$ denote the set of all sequences $B_i$, $i\in\Z^{+}$, of blocks in $\B$.
Declare two sequences to be equivalent, written $B_i\sim C_i$, if and only if their corresponding Borromean rays are equivalent by some diffeomorphism of $\R^3$ (not necessarily orientation preserving).

\begin{corollary}\label{uncount_borr_rays}
The set of equivalence classes $\mathcal{S}/\sim$ is uncountable.
\end{corollary}

\begin{proof}
By Theorem~\ref{borr_rays_thm}, the equivalence class of any given sequence is countable.
\end{proof}

A multiray $\tau\subset \R^3$ is \textbf{achiral} provided there exists a diffeomorphism $f:\pa{\R^3,\tau}\to\pa{\R^3,\overline{\tau}}$ preserving orientation of $\R^3$.
Otherwise, $\tau$ is \textbf{chiral}.
Equivalently, $\tau$ is \textbf{achiral} provided there exists a diffeomorphism $f:\pa{\R^3,\tau}\to\pa{\R^3,\tau}$ reversing orientation of $\R^3$.

\begin{corollary}\label{chiral_cor}
Let $\pa{\R^3,\tau}=D^3 B_1 B_2 B_3\cdots$ for a sequence $B_i$, $i\in\Z^{+}$, of blocks in $\B$.
Then, $\tau$ is achiral if and only if (i) there exists a block $C=C_1 C_2 \cdots C_k$ where $k\in\Z^{+}$ and each $C_i\in\B$, and (ii) a tail of the sequence $B_i$ equals one of the following:
\begin{equation}\label{Cper}
 C\overline{C} C \overline{C} C\overline{C} \cdots \quad \tn{or} \quad C C^{\ast} C C^{\ast} C C^{\ast} \cdots
\end{equation}
In particular, if $\tau$ is achiral, then $B_i$ is eventually periodic.
So, at most countably many achiral $\tau$ arise this way.
\end{corollary}

\begin{proof}
By Theorem~\ref{borr_rays}, $\tau$ is achiral if and only if (i) $B_i$ and $\overline{B_i}$ have identical tails,
or (ii) $B_i$ and $B_i^{\ast}$ have identical tails.
So, if a tail of $B_i$ has the form~\eqref{Cper}, then $\tau$ is achiral.
Conversely, suppose $B_i$ and $\overline{B_i}$ have identical tails (the other case is similar).
Then, there exists $n\in\Z$ and $N\in\Z^{+}$ such that
\begin{equation}\label{shift}
B_i=\overline{B_{i+n}} \tn{ for all } i\geq N
\end{equation}
As $B_N\neq\overline{B_N}$, we get $n\neq0$.
Without loss of generality, assume $n>0$ (otherwise, apply bar to~\eqref{shift}).
Note that:
\[
B_{N+n}=\overline{\overline{B_{N+n}}}=\overline{B_{N}}
\]
where the second equality used~\eqref{shift}.
Repeating this argument, we get that if $m\geq0$ and $0\leq j<n$, then:
\[
B_{N+mn+j}=
\begin{dcases*}
B_{N+j}	&	if $m$ is even\\
\overline{B_{N+j}}	&	if $m$ is odd
\end{dcases*}
\]
Therefore, $C=B_N B_{N+1} \cdots B_{N+n-1}$.
\end{proof}

\begin{example}\label{achiral_br}
For each $m\in\Z^{+}$, define:
\[
C_m :=A^{m}=\underbrace{A A \cdots A}_{m \tn{ times}}
\]
and define:
\[
\pa{\R^3,\tau_m}:=D^3 C_{m} \overline{C_m} C_{m} \overline{C_m} C_{m} \overline{C_m} \cdots
\]
By Corollary~\ref{chiral_cor}, $\tau_m$ is achiral.
By Theorem~\ref{borr_rays_thm}, $m$ is a diffeomorphism invariant of $\tau_m$.
So, $\tau_m$, $m\in\Z^{+}$, is a countably infinite family of achiral Borromean rays, pairwise distinct up to diffeomorphism.
\end{example}

\subsection{Borromean Hyperplanes}\label{borr_hyp}

A \textbf{hyperplane} is a smooth, proper embedding of $\R^{n-1}$ in $\R^{n}$.
A \textbf{multiple hyperplane} $H\subset \R^n$ is a smooth, proper embedding of a disjoint union of (at most countably many) copies of $\R^{n-1}$.
The basic invariant of $H$ is an associated tree $T(H)$.
The vertices of $T(H)$ are the components of $\R^n - H$.
Two vertices are adjacent provided their closures in $\R^n$ share a component of $H$.
Figure~\ref{hyp_graphs} depicts some multiple hyperplanes and their trees.
\begin{figure}[h!]
    \centerline{\includegraphics[scale=1.0]{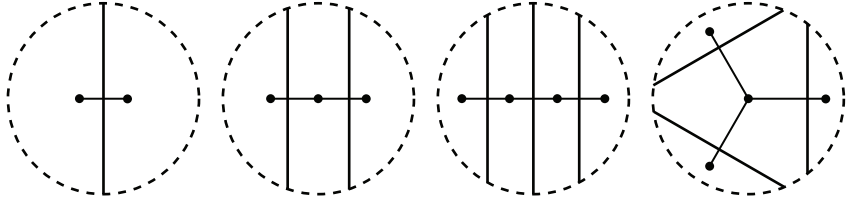}}
    \caption{Multiple hyperplanes and their associated trees.}
    \label{hyp_graphs}
\end{figure}

Multiple hyperplanes in $\R^n$, $n\neq3$, are classified by their associated trees~\cite[\S9]{cks}.
More precisely, if $H$ and $H'$ are multiple hyperplanes in $\R^n$, $n>3$, and the trees $T(H)$ and $T(H')$ are isomorphic,
then there is a diffeomorphism $f:\pa{\R^n,H}\to\pa{\R^n,H'}$ that preserves orientation of $\R^n$.
For $n=2$, these trees are naturally planar (i.e., the edges incident with a given vertex are cyclically ordered).
The result then holds provided $T(H)$ and $T(H')$ are isomorphic as planar trees.\\

Let $\H^3$ denote Klein's model of hyperbolic $3$-space.
Namely, $\H^3$ is the open unit $3$-disk in $\R^3$, and a \textbf{hyperbolic hyperplane} is the nonempty intersection of $\H^3$ with an affine plane in $\R^3$.
A \textbf{hyperbolic multiple hyperplane} is a properly embedded submanifold of $\H^3$, each component of which is a hyperbolic hyperplane.
The multiple hyperplanes in Figure~\ref{hyp_graphs} are hyperbolic.
A multiple hyperplane $H\subset \R^3$ is \textbf{unknotted} provided there exists a diffeomorphism $f:\pa{\R^3,H}\to\pa{\H^3,L}$ where $L$ is a hyperbolic multiple hyperplane.

\begin{lemma}
Let $H$ and $H'$ be unknotted multiple hyperplanes in $\R^3$.
There exists a diffeomorphism $f:\pa{\R^3,H}\to\pa{\R^3,H'}$ preserving orientation of $\R^3$ if and only if $T(H)$ and $T(H')$ are isomorphic trees.
\end{lemma}

\begin{proof}
The classification proof for multiple hyperplanes in $\R^n$, $n> 3$, applies to \emph{unknotted} multiple hyperplanes in $\R^3$~\cite[\S9]{cks}.
\end{proof}

Given a tree $T$ that is at most countable and is not necessarily locally finite, it is not difficult to construct a hyperbolic multiple hyperplane
$L\subset\H^3$ such that $T(L)$ is isomorphic to $T$.
Hence, unknotted multiple hyperplanes in $\R^3$ are classified, up to diffeomorphism, by isomorphism classes of such trees.
Up to isomorphism, there is a unique tree with $n$ vertices for $n\leq 3$.
For $n=4$, there are two: a linear tree and the $3$-prong (see Figure~\ref{hyp_graphs}).\\

Let $\R^3_+ :=\cpa{(x,y,z) \mid z\geq 0}$ denote closed upper half space.
Throughout this section, $\zeta\subset\R^3$ denotes the ray $\cpa{(0,0)}\times[1,\infty)$ in the positive $z$-axis.\\

If $r\subset \R^3$ is a multiray with $1\leq n \leq \infty$ components, then $\nu r$ denotes a smooth, closed regular neighborhood of $r$ in $\R^3$~\cite[\S3]{cks} (see also Hirsch~\cite{hirsch_62}).
If $r_i$ is a component of $r$, then $\nu r_i$ denotes the component of $\nu r$ containing $r_i$.
Basic properties of $\nu r$ include:
\begin{enumerate}\setcounter{enumi}{\value{equation}}
\item\label{mnr1} For each $r_i$, there is a diffeomorphism of pairs $\varphi_i:\pa{\nu r_i,r_i}\to\pa{\R^3_+,\zeta}$.
\item\label{mnr2} $\nu r$ is unique up to ambient isotopy of $\R^3$ relative to $r$.
\item\label{mnr3} The boundary of $\nu r$, denoted $\partial \nu r$, is an $n$ component multiple hyperplane.
\item\label{mnr4} The tree $T(\partial \nu r)$ is an $n$-prong. 
\setcounter{equation}{\value{enumi}}
\end{enumerate}

We wish to show that each multiray $r\subset\R^3$ is essentially determined by the multiple hyperplane $H=\partial \nu r$.

\begin{lemma}\label{ztoz}
Let $\alpha:\R^3_+\to\R^3_+$ be a diffeomorphism.
Then, $\alpha$ is isotopic, relative to a neighborhood of $\partial \R^3_+$ in $\R^3_+$, to a diffeomorphism of pairs $\beta:\pa{\R^3_+,\zeta}\to \pa{\R^3_+,\zeta}$.
\end{lemma}

\begin{proof}
Let $\tau\subset\R^3$ be the ray $[1,\infty)\times\cpa{(0,0)}$ in the positive $x$-axis.
Let $S$ denote the points of norm $\geq 1$ in the first quadrant of the $xz$-plane.
We identify $S$ with $[0,1]\times[1,\infty)$ so that $\tau = \cpa{0}\times[1,\infty)$ and $\zeta=\cpa{1}\times[1,\infty)$.
All isotopies of $\R^3_+$ will be ambient and relative to a neighborhood of $\partial \R^3_+$ in $\R^3_+$.
Isotoped subsets will be called by their original names.
It suffices to construct an isotopy of $\R^3_+$ that carries $\alpha(\zeta)$ to $\zeta$.
Let $\pi:\R^3_+ \to \R^2\times\cpa{0}$ be orthogonal projection.
The submanifold $\alpha(S)$ permits construction of an isotopy of $\R^3_+$ that carries $\alpha(\zeta)$ close enough to $\alpha(\tau)$ so that $\left.\pi\right|\alpha(\zeta)$ is an embedding.
This is possible since $\alpha(\tau)$ is properly embedded.
Next, by integrating a suitable vector field of the form $v(x,y,z)=(0,0,\ast)$, we get an isotopy carrying $\alpha(\zeta)$ into $\R^2\times\cpa{1}$.
Rays do not knot in $\R^2$~\cite[Thm.~9.13]{cks}.
So, there is an ambient isotopy $F$ of $\R^2\times\cpa{1}$ that carries $\alpha(\zeta)$ to a straight ray.
Use a small tube about $\R^2\times\cpa{1}$ in $\Int \R^3_+$ and a suitable bump function to extend $F$ to an isotopy of $\R^3_+$.
The rest is elementary.
\end{proof}

\begin{remarks}
\noindent\begin{enumerate}[label=(\arabic*),leftmargin=*]\setcounter{enumi}{0}
\item Lemma~\ref{ztoz} is very different from `uniqueness of regular neighborhoods'.
Let $K$ be a smooth subcomplex of a smooth manifold $M$.
It is not true, in general, that each orientation preserving diffeomorphism $h:\nu K\to \nu K$ is isotopic to a diffeomorphism $\pa{\nu K,K}\to\pa{\nu K,K}$.
For a simple counterexample, let $K$ be a bouquet of three circles embedded in $M=\R^2$ in such a way that no circle of $K$ is inside another.
Then, $\nu K\subset M$ is a smooth, compact $2$-disk with three holes.
Let $C\subset \Int \nu K$ be a simple closed curve such that two boundary components of $\nu K$ are inside $C$.
Let $h:\nu K \to \nu K$ be a Dehn twist about $C$.
Then, $h$ is not isotopic to a diffeomorphism $\pa{\nu K,K}\to\pa{\nu K,K}$.
Otherwise, $h$ would be isotopic to the identity, which is false~\cite[pp.~239--247]{farbmargalit}.
\item Counterexamples exist even when $K$ is a smooth submanifold. We are indebted to Bob Gompf for these examples.
Let $X$ be a simply-connected, closed, symplectic $4$-manifold with positive signature, denoted $\sigma(X)>0$, and $b_2^+>1$.
(Many such manifolds are known---even K\"{a}hler examples have been around for several decades.)
As $X$ is symplectic, $X$ is smooth and oriented, and $b_2^+$ is odd.
Let $K$ be $X$ blown up $\sigma(X)$ times.
Then, $\sigma(K)=0$ and $K$ is homeomorphic to $Z:=\sharp_m (\C P^2 \sharp\overline{\C P}^2)$ for some $m>1$ (and very large in practice).
Since $K$ is symplectic, $K$ has nonvanishing Seiberg-Witten invariants, denoted $SW(K)\neq 0$.
However, $\overline{K}$ (= $K$ with reversed orientation) splits off a $\C P^2$ summand (from the blowup of $X$) and has $b_2^+>1$.
Therefore, $SW(\overline{K})=0$ and $K$ admits no orientation reversing self diffeomorphism.
Fix any $n\geq 2$.
Then, $M:=K\times S^n$ is a smooth, closed, oriented manifold containing $K=K\times\cpa{p}$, and $\nu K \subset M$ is identified with $K\times D^n$.
Now, there is a smooth $h$-cobordism $W$ between $K$ and $Z$.
So, $W\times S^{n-1}$ is smoothly a product, and, working relative to boundary, $W\times D^n$ is smoothly a product.
Hence, there is a diffeomorphism $K\times D^n \to Z \times D^n$.
Let $\alpha$ be an orientation reversing self diffeomorphism of $Z$ (e.g., $\alpha$ permutes the summands and is otherwise the identity). Let $\beta$ be an orientation reversing self diffeomorphism of $D^n$.
Then, $k:=\alpha\times\beta$ is an orientation preserving self diffeomorphism of $Z \times D^n$,
and $k_{\ast}$ is multiplication by $-1$ on $H_4\pa{Z \times D^n  ;\Z}$.
Let $h$ be the corresponding orientation preserving self diffeomorphism of $\nu K$.
On $H_4\pa{\nu K  ; \Z}$, $h_{\ast}$ is multiplication by $-1$.
So, $h$ is not isotopic to a diffeomorphism $\pa{\nu K,K}\to\pa{\nu K,K}$.
Otherwise, we get a forbidden orientation reversing self diffeomorphism of $K$.
\item An alternative approach to proving Lemma~\ref{ztoz} uses the following lemma together with some collaring arguments.
\end{enumerate}
\end{remarks}

\begin{lemma}
Let $\gamma:\R^n_+\to\R^n_+$ be an orientation preserving diffeomorphism.
Then, $\gamma$ is isotopic to the identity.
\end{lemma}

\begin{proof}
As $\gamma\pa{\R^{n-1}\times\cpa{0}}=\R^{n-1}\times\cpa{0}$,
the well-known proof for a diffeomorphism $\R^n\to\R^n$ applies (see Milnor~\cite[p.~34]{milnor97}).
\end{proof}

The following lemma is useful.

\begin{lemma}\label{unknot_equiv}
Let $H\subset \R^n$ be a hyperplane. Let $X$ and $Y$ denote the closures in $\R^n$ of the components of $\R^n-H$.
The following are equivalent:
\begin{enumerate}\setcounter{enumi}{\value{equation}}
\item\label{eisotopy} There is an ambient isotopy $F$ of $\R^n$ such that $F_1(H)=\R^{n-1}\times\cpa{0}$.
\item\label{ediffeo} There is a diffeomorphism $f:\pa{\R^n,H}\to\pa{\R^n,\R^{n-1}\times\cpa{0}}$.
\item\label{ehalves} There are diffeomorphisms $\varphi:X\to\R^n_+$ and $\psi:Y\to\R^n_+$.
\setcounter{equation}{\value{enumi}}
\end{enumerate}
\end{lemma}

\begin{proof}
Only two implications require proof.\\

\ref{ediffeo} $\Rightarrow$ \ref{eisotopy}: We may assume $f$ preserves orientation of $\R^n$ (otherwise, compose $f$ with reflection through $\R^{n-1}\times\cpa{0}$).
By Milnor~\cite[p.~34]{milnor97}, $f$ is isotopic to the identity.\\

\ref{ehalves} $\Rightarrow$ \ref{ediffeo}: Let $\R^n_-$ denote closed lower half space.
Replacing $\varphi$ and $\psi$ with their compositions with appropriate reflections, we can and do assume $\varphi:X\to\R^n_+$ and $\psi:Y\to\R^n_-$ are orientation preserving diffeomorphisms.
Let $\mu=\left.\psi\circ\varphi^{-1}\right|\R^{n-1}\times\cpa{0}$, an orientation preserving automorphism of $\R^{n-1}\times\cpa{0}$.
Let $\eta$ be the orientation preserving automorphism of $\R^{n-1}$ given by the following composition where $\tn{pr}$ is the obvious projection:
\[
	\R^{n-1} \hookrightarrow \R^{n-1}\times\cpa{0} \stackrel{\mu}{\rightarrow} \R^{n-1}\times\cpa{0} \stackrel{\tn{pr}}{\rightarrow} \R^{n-1}
\]
Let $\tn{Id}$ be the identity map on $[0,\infty)$.
Then, $\eta\times\tn{Id}$ is an orientation preserving automorphism of $\R^n_+$.
Replacing $\varphi$ with $(\eta\times\tn{Id})\circ\varphi$, we can and do further assume $\left.\varphi\right|H=\left.\psi\right|H$.
Define $h:\R^n \to \R^n$ by $h(p)=\varphi(p)$ if $p\in X$ and $h(p)=\psi(p)$ if $p\in Y$.
Then, $h$ is an orientation preserving autohomeomorphism of $\R^n$. By construction, $h$ is smooth on $X$ and $h$ is smooth on $Y$.
Using collaring uniqueness~\cite[Thm.~8.1.9]{hirsch}, we may adjust $h$ (by isotoping $\varphi$ and $\psi$ near $H$ and relative to $H$) to obtain the desired diffeomorphism $f$.
\end{proof}

\begin{corollary}\label{rtor}
Let $r$ and $r'$ be multirays in $\R^3$.
Let $h:\nu r \to \nu r'$ be a diffeomorphism.
Let $H=\partial \nu r$ and $H'=\partial \nu r'$.
Then, $h$ is isotopic, relative to a neighborhood of $H$ in $\nu r$, to a diffeomorphism of pairs $g:\pa{\nu r,r}\to \pa{\nu r',r'}$.
\end{corollary}

\begin{proof}
It suffices to consider $\left.h\right|\nu r_i$ where $r_i$ is a component of $r$.
Reindex the components of $r'$ so that $h(\nu r_i)=\nu r'_i$.
Let $\varphi_i:\pa{\nu r_i,r_i}\to\pa{\R^3_+,\zeta}$ and $\psi_i:\pa{\nu r'_i,r'_i}\to\pa{\R^3_+,\zeta}$ be diffeomorphisms.
Lemma~\ref{ztoz} yields an isotopy \hbox{$F:\R^3_+ \times[0,1]\to\R^3_+$}, relative to a neighborhood of $\partial \R^3_+$,
such that $F_0=\psi_i h \varphi^{-1}_i$ and $F_1 (\zeta)=\zeta$.
Then, $\psi^{-1}_i \circ F \circ (\varphi_i\times\tn{Id})$ is the desired isotopy of $\left.h\right|\nu r_i$.
\end{proof}

\begin{corollary}\label{krikh}
If $r\subset \R^3$ is a knotted multiray, then $H:=\partial \nu r$ is a knotted multiple hyperplane. 
\end{corollary}

\begin{proof}
Suppose $H$ is unknotted.
Then, there is a diffeomorphism $h:\pa{\R^3,H}\to\pa{\H^3,L}$ where $L$ is a hyperbolic multiple hyperplane.
As $T(H)$ is an $n$-prong, so is $T(L)$.
Without loss of generality, the origin of $\H^3$ does not lie in $h(\nu r)$.
For each component $r_i$ of $r$, let $C_i=h(\nu r_i)$ and let $L_i=\partial C_i$.
Let $p_i$ be the point of $L_i$ closest to the origin in $\H^3$ (for the euclidean metric).
Let $\sigma_i\subset C_i$ be the radial ray in $\H^3$ with initial point $p_i$.
Let $\tau_i\subset\Int C_i$ be the radial ray in $\sigma_i$ that is half as long as $\sigma_i$ (for the euclidean metric).
Notice that $C_i$ is a smooth, closed regular neighborhood of $\tau_i$ in $\H^3$.
Let $\tau\subset\H^3$ be the radial multiray with components $\tau_i$.
Notice that $h(\nu r)$ is a smooth, closed regular neighborhood of $\tau$ in $\H^3$. 
As in Corollary~\ref{rtor}, we may isotop $h$ to a diffeomorphism $\pa{\R^3,r}\to\pa{\H^3,\tau}$.
But, this implies $r$ is unknotted, a contradiction.
\end{proof}

\begin{proposition}\label{rays_planes}
Let $\tau$ and $\tau'$ be multirays in $\R^3$, each containing $n$ components where $1\leq n\leq\infty$.
Let $H=\partial \nu \tau$ and let $H'=\partial \nu \tau'$.
If $f:\pa{\R^3,\tau}\to\pa{\R^3,\tau'}$ is a diffeomorphism, then $f$ is isotopic relative to $\tau$ to a diffeomorphism $g:\pa{\R^3,H}\to\pa{\R^3,H'}$.
Conversely, suppose $g:\pa{\R^3,H}\to\pa{\R^3,H'}$ is a diffeomorphism.
If any of the following conditions are met, then $g$ is isotopic relative to $G:=\R^3 -\Int \nu \tau$ to a diffeomorphism $f:\pa{\R^3,\tau}\to\pa{\R^3,\tau'}$.
\begin{enumerate}\setcounter{enumi}{\value{equation}}
\item\label{cond1} $g(\nu \tau)=\nu \tau'$.
\item\label{cond2} $n\geq 2$.
\item\label{cond3} $n=1$, and $\tau$ or $\tau'$ is knotted.
\setcounter{equation}{\value{enumi}}
\end{enumerate}
\end{proposition}

\begin{remark}
The case $n=1$ where $\tau$ and $\tau'$ are unknotted is exceptional for trivial reasons.
For example, let $\tau$ and $\tau'$ both equal $\zeta$.
Let $\nu \tau$ and $\nu \tau'$ both equal $\R^3_+$, so $H=H'=\R^2\times\cpa{0}$.
Then, $g(x,y,z)=(-x,y,-z)$ is an orientation preserving diffeomorphism of $\R^3$ sending $H$ to $H'$.
However, $g$ is not isotopic relative to $H$ to a diffeomorphism $\pa{\R^3,\tau}\to\pa{\R^3,\tau'}$. 
\end{remark}

\begin{proof}[Proof of Proposition~\ref{rays_planes}]
The forward implication is immediate by ambient uniqueness of closed regular neighborhoods~\cite[\S3]{cks}.
Next, let $g:\pa{\R^3,H}\to\pa{\R^3,H'}$ be a diffeomorphism.
First, assume condition~\ref{cond1}.
Then, the result is immediate by Corollary~\ref{rtor}.
Second, assume condition~\ref{cond2}.
As $T(H)$ and $T(H')$ are both $n$-prongs, condition~\ref{cond1} is satisfied and the result follows.
Third, assume condition~\ref{cond3}.
Without loss of generality, assume $\tau$ is knotted (otherwise, consider $g^{-1}$).
Let $X=\nu \tau$ and let $Y=\R^3-\Int \nu \tau$.
By Corollary~\ref{krikh}, $H$ is knotted.
By Lemma~\ref{unknot_equiv}, $Y\not\approx\R^3_+$.
As $\nu \tau' \approx \R^3_+$, we see that condition~\ref{cond1} is satisfied and the result follows.
\end{proof}

Proposition~\ref{rays_planes} permits us to translate results on knotted multirays in $\R^3$ to results on knotted multiple hyperplanes in $\R^3$.
By McPherson~\cite{mcpherson}, there exist uncountably many knot types of a ray: so there exist uncountably many knot types of a hyperplane.
By Fox and Harrold~\cite{foxharrold}, there exist uncountably many knot types of two component multirays with unknotted components
(see Fox and Artin~\cite[p.~988]{foxartin} for a nice example): so there exist uncountably many knot types of two component multiple hyperplanes with unknotted components.\\

Proposition~\ref{rays_planes} is proved via ambient isotopies, so it also yields results on chirality.
Let $H\subset\R^3$ be a multiple hyperplane.
We say $H$ is \textbf{achiral} provided there is a diffeomorphism $\pa{\R^3,H}\to\pa{\R^3,H}$ that reverses orientation of $\R^3$.
Otherwise, $H$ is \textbf{chiral}.
For example, it is an exercise to show that each unknotted multiple hyperplane $H\subset \R^3$ is achiral.\\

A multiple hyperplane $H$ in $\R^3$ forms \textbf{Borromean hyperplanes} provided $H$ is knotted, but any two components of $H$ form an unknotted multiple hyperplane.
Proposition~\ref{rays_planes} implies that if $\tau\subset\R^3$ forms Borromean rays, then $\partial \nu r$ forms Borromean hyperplanes.
Thus, we obtain our second main result.

\begin{theorem}
There exist uncountably many Borromean hyperplanes, pairwise distinct up to diffeomorphism of $\R^3$.
There exists a countably infinite family of achiral Borromean hyperplanes, pairwise distinct up to diffeomorphism of $\R^3$.
\end{theorem}

\begin{proof}
Immediate by Proposition~\ref{rays_planes}, Corollary~\ref{uncount_borr_rays}, and Example~\ref{achiral_br}.
\end{proof}

\end{document}